\documentclass[11pt,reqno]{amsart}
\usepackage[colorlinks=true,
pdfstartview=FitV, linkcolor=cyan, citecolor=magenta,
urlcolor=blue]{hyperref}
\usepackage{amsmath,amsfonts,latexsym,amssymb}
\usepackage{mathrsfs}
\usepackage[latin1]{inputenc}
\usepackage[T1]{fontenc}
\usepackage{ae,aecompl}
\usepackage{braket}
\usepackage{comment}
\usepackage{color}
\usepackage{graphicx}
\usepackage{subfig}
\usepackage{bbm}

\usepackage{enumitem}

\newtheorem{theorem}{Theorem}[section]
\newtheorem{lemma}[theorem]{Lemma}
\newtheorem{proposition}[theorem]{Proposition}
\newtheorem{corollary}[theorem]{Corollary}

\newtheorem*{thm*}{\protect\theoremname}
\theoremstyle{definition}

\newtheorem*{cor*}{\protect\corollaryname}

\renewcommand{\geq}{\geqslant}

\long\def\@savemarbox#1#2{\global\setbox#1\vtop{\hsize\marginparwidth 
  \@parboxrestore\tiny\raggedright #2}}
\newcommand{\PSL}{\mathsf{PSL}}
\newcommand{\PGL}{\mathsf{PGL}}

\newcommand{\abs}[1]{\left|#1\right|}

\newcommand{\Bc}{\mathcal B}

\newcommand{\Oc}{\mathcal O}

\DeclareMathOperator{\Hb}{\mathbb{H}}

\DeclareMathOperator{\Pb}{\mathbb{P}}
\DeclareMathOperator{\Rb}{\mathbb{R}}

\DeclareMathOperator{\Psf}{\mathsf{P}}

\DeclareMathOperator{\Fc}{\mathcal{F}}

\DeclareMathOperator{\SL}{\mathsf{SL}}
\newcommand{\norm}[1]{\left\|#1\right\|}

\usepackage{ mathrsfs }

\providecommand{\corollaryname}{Corollary}
\providecommand{\theoremname}{Theorem}

\begin{document}

\title{Geometry and Dynamics of Transverse Groups}
\author[Canary]{Richard Canary}
\address{University of Michigan}
\author[Zhang]{Tengren Zhang}
\address{National University of Singapore}
\author[Zimmer]{Andrew Zimmer}
\address{University of Wisconsin-Madison}
\thanks{Canary was partially supported by grant  DMS-2304636 from the National Science Foundation.
Zhang was partially supported by the NUS-MOE grant R-146-000-270-133. 
Zimmer was partially supported by a Sloan research fellowship and grants DMS-2105580 and DMS-2104381 from the
National Science Foundation.}


\maketitle

\tableofcontents

\section{Introduction}

In this paper we survey the geometry and dynamics of transverse subgroups of semi-simple Lie groups. Examples of transverse groups include 
all discrete subgroups of rank 1 Lie groups, as well as Anosov and relatively Anosov subgroups of higher rank semi-simple Lie groups.
They were first introduced by Kapovich,  Leeb  and Porti \cite{KLP}, who
called them regular, antipodal groups. This paper will focus on the theory of Patterson-Sullivan measures and its consequences for the 
geometry and dynamics of such groups. 

Our motivation will be provided by two classical theorems about closed hyperbolic surfaces.

\begin{theorem}[Hedlund \cite{hedlund}]
If $X$ is a closed hyperbolic surface, then the geodesic flow on $T^1X$ is ergodic with respect to the Liouville measure, i.e. if $A\subset T^1X$ is a measurable set which is invariant under the flow, then $A$ either has zero or full Liouville measure.
\end{theorem}

\begin{theorem}[Huber \cite{huber}]\label{thm:Huber} If $X$ is a closed hyperbolic surface,  and $N_X(T)$ denotes the number of closed geodesics on $X$ of length at most $T$,
then
$$N_X(T)\sim \frac{e^T}{T}$$
in the sense that $\lim_{T\to\infty} \frac{TN_X(T)}{e^T}=1$.
\end{theorem}

Both theorems generalize to the setting of closed negatively curved manifolds (by work of Anosov \cite{anosov} and Margulis \cite{margulis}).
However, if $X$ is an infinite volume hyperbolic surface with finitely generated fundamental
group, then the geodesic flow on $T^1X$ is not ergodic with respect to the Liouville measure. Moreover, with respect to the Liouville measure almost every flow line is non-recurrent, i.e. returns to any compact set only finitely often.

Patterson--Sullivan theory provides a way to put a measure on the geodesic flow which is supported on the non-wandering portion of the geodesic flow
and captures its chaotic behavior. In particular, in the case mentioned above, the geodesic flow is ergodic with respect to this measure. Moreover, one can obtain asymptotic results on the number of closed geodesics of length at most $T$.

We will begin with a brief survey of the classical theory of Patterson--Sullivan measures for hyperbolic surfaces. We will then discuss our  work \cite{CZZ2,CZZ3,CZZ4} on Patterson--Sullivan theory for transverse subgroups of semi-simple Lie groups. 
We will also discuss its implications
in the better understood settings of Anosov and relatively Anosov groups. For simplicity of exposition, we will reduce to the setting of transverse
subgroups of $\mathsf{SL}(d,\mathbb R)$.

In a final section, we will survey recent work of Kim--Oh--Wang \cite{KOW}, Blayac--Canary--Zhu--Zimmer \cite{BCZZ1,BCZZ2} and Kim--Oh \cite{KO}.

\medskip\noindent
{\bf Acknowledgements:} The authors would like to thank Hee Oh for helpful comments on an earlier version of this manuscript. Richard Canary
would like to thank the organizers of the conference on ``New Aspects of Teichm\"uller theory'' for the opportunity to speak at the conference.

\section{Fuchsian groups}

A \emph{Fuchsian group}  is a discrete subgroup of $\mathsf{PSL}(2,\mathbb R)$, which we regard as the group of orientation-preserving isometries of 
the upper half-plane model for the hyperbolic plane $\mathbb H^2$. We first describe Patterson--Sullivan theory for Fuchsian groups. This special case will serve as a model for the more general theory described in later sections.  

We begin by studying the Hopf parametrization of $T^1\mathbb H^2$ and the description of the geodesic flow and the
action of $\mathsf{PSL}(2,\mathbb R)$ on $T^1\mathbb H^2$ in this parametrization.
The Hopf parameterization identifies $T^1\mathbb H^2$ with $\partial\mathbb H^{(2)}\times \mathbb R$, where $\partial\mathbb{H}^{(2)}$ denotes the set of distinct pairs of points in $\partial\mathbb{H}$. Explicitly, if $w,z\in\partial\mathbb{H}^2$ are distinct, let $r_{w,z}:\mathbb R\to\mathbb{H}^2$ be the unit speed geodesic from $w$ to $z$ so that $r_{w,z}(0)$ lies on the horocycle $H_z$ based at $z$ passing through $i$. Then $(w,z,t)$ is identified with $r_{w,z}'(t)$. In the Hopf parametrization, the geodesic flow $\{\varphi_s\}_{s\in\mathbb R}$ has the simple form
$$\varphi_s(w,z,t)=(w,z,t+s).$$

We define the \emph{Busemann function} $b:\partial\mathbb H^2\times\mathbb H^2\to\mathbb R$ by letting $b(z,x)$ be the signed distance between
$H_z$ and the horocycle $H_{z,x}$ based at $z$ passing through $x$,  where the distance is positive if and only if $H_{z,x}$ lies inside  of the horodisk
bounded by $H_z$.  One may check that if $\{y_n\}$ is a sequence in $\mathbb H^2$ converging to $z\in \partial\mathbb{H}^2$,
then 
$$b(z,x)=\lim_{y_n\to z} d(y_n,i)-d(y_n,x).$$
We then define the \emph{Busemann cocycle}
$B:\mathsf{PSL}(2,\mathbb R)\times \partial\mathbb H^2\to\mathbb R$ by letting
$B(\gamma,z)$ be the signed distance between $\gamma(H_z)$ and $H_{\gamma(z)}$, where the distance is positive if and only if $\gamma(H_z)$ lies inside  of the horordisk
bounded by $H_{\gamma(z)}$.
Notice that $B(\gamma,z)=b(\gamma(z),\gamma(i))=-b(z,\gamma^{-1}(i)).$
Furthermore, the action of $\gamma\in\mathsf{PSL}(2,\mathbb R)$ in the Hopf parametrization is then  given
by
$$\gamma(w,z,t)=\Big(\gamma(w),\gamma(z), t+B(\gamma,z)\Big).$$

One may identify $\partial\mathbb H^2$ with $S^1$,  by identifying it with $T^1_i\mathbb H^2$. 
We scale the induced Lebesgue measure on $\partial\mathbb H^2$ to obtain a probability measure $d\lambda$ that we loosely refer to as the
Lebesgue measure on $\partial\mathbb H^2$.
We then define  a Gromov product
$G:\partial\mathbb H^2\times \partial\mathbb H^2\to \mathbb R$ by letting 
$$G(w,z)=b(w,y)+b(z,y)$$
where $y$ is any point on the geodesic joining $w$ to $z$.
One may then define the measure $\tilde m$ on $T^1\mathbb H^2$ to be
$$d\tilde m(w,z,t)=e^{G(w,z)}d\lambda d\lambda ds$$
where $ds$ is Lebesgue measure on $\mathbb R$. The Gromov product has been chosen exactly so that $\tilde m$ is invariant under the transitive action
of $\mathsf{PSL}(2,\mathbb R)$, and hence $\tilde m$  is a scalar multiple of the usual Lebesgue measure.

Let $\Gamma$ be a Fuchsian group, and denote $X_\Gamma:=\Hb^2/\Gamma$. The \emph{limit set} $\Lambda(\Gamma)$ of a Fuchsian group $\Gamma$ is the set of accumulation points of the orbit $\Gamma(i)$ in $\partial\mathbb H^2$. We say that $\Gamma$ is \emph{non-elementary} if $\Lambda(\Gamma)$ is infinite. 
When $\Gamma$ is non-elementary, one may define $\Lambda(\Gamma)$ dynamically as the smallest closed, non-empty, $\Gamma$-invariant subset of $\partial\mathbb H^2$. 

The measure $\tilde m$ descends to a flow invariant measure on $T^1X_\Gamma$ called the Liouville measure. 
 However, if $\Lambda(\Gamma)$ is not all of $\partial\mathbb H^2$, then the geodesic flow on $T^1X_\Gamma$ is not ergodic with respect to the Liouville measure.
 If $\Gamma$ is finitely generated, then the converse holds (since if $\Gamma$ is finitely generated and $\Lambda(\Gamma)=\partial\mathbb H^2$ then $\Gamma$ is a
 lattice).

A point $y\in T^1X_\Gamma$ is \emph{wandering} if it admits a neighborhood $V$ and $T>0$ so that for all $t\ge T$,
$\phi_t(V)$ is disjoint from $V$. One may  show that if $\pi:T^1\mathbb H^2\to T^1X_\Gamma$ is the covering map,
then $\pi(w,z,t)$ is non-wandering for the geodesic flow on $T^1X_\Gamma$ if and only both $w,z\in\Lambda(\Gamma)$.
We then denote the non-wandering portion of the geodesic flow by
$$T^1X_\Gamma^{nw}=\left(\Lambda(\Gamma)^{(2)}\times\mathbb R\right)/\Gamma,$$
where $\Lambda(\Gamma)^{(2)}$ denotes the set of distinct pairs of points in $\Lambda(\Gamma)$.

If $\Gamma$ is finitely generated and $\Lambda(\Gamma)\ne\partial\mathbb H^2$, then the limit set has Lebesgue measure zero,
so $T^1X_\Gamma^{nw}$ has measure zero in the Lebesgue measure on $T^1X_\Gamma$.
We will construct a measure on $T^1X_\Gamma$ which is flow-invariant and supported
on $T^1X_\Gamma^{nw}$. The above discussion hopefully suggests that we should begin by constructing a measure  on $\partial\mathbb H^2$ which
is supported on the limit set.

We define the \emph{Poincar\'e series} of a non-elementary Fuchsian group $\Gamma$ to be
$$P_\Gamma(s)=\sum_{\gamma\in\Gamma} e^{-s d(i,\gamma(i))}.$$
The root test  tells us that there exists $\delta\in[0,\infty]$ such that $P_\Gamma(s)$
converges if $s>\delta$ and diverges if $s<\delta$. We call $\delta=\delta(\Gamma)$
the \emph{critical exponent} of $\Gamma$. 
More geometrically, one may check that $\delta$ is the exponential growth rate of the orbit, i.e. if 
$$\Gamma(R)=
\{\gamma\in\Gamma: d(i,\gamma(i))\le R\}$$
then
$$\delta (\Gamma)=\limsup_{R\to\infty} \frac{\log \#\Gamma(R)}{R}.$$
It follows immediately from this description that $\delta(\Gamma)\le 1$ and  that $\delta(\Gamma)=1$ if $X_\Gamma$
is a closed surface (since then the orbit has the same exponential growth rate as $\mathbb H^2$.)

Notice that with respect to the identification of $\partial\mathbb H^2$ with $S^1$, the Lebesgue measure $\lambda$ satisfies the following quasi-invariance property: for all $\gamma\in\Gamma$,
$$\frac{d\gamma_*(\lambda)}{d\lambda}(z)=\Big|\big(\gamma^{-1}\big)'(z)\Big| \quad \text{or}\quad \frac{d\gamma^*(\lambda)}{d\lambda}(z)=\Big|\gamma'(z)\Big|$$
for almost every $z\in\partial\Hb^2$.
One may  then check that
$$\Big|\big(\gamma^{-1}\big)'(z)\Big| =e^{-B(\gamma^{-1},z)}=e^{b(z,\gamma(i))}.$$
We will say that a probability measure $\mu$ on $\Lambda(\Gamma)$ is a \emph{Patterson--Sullivan measure} for $\Gamma$ (of dimension $\delta(\Gamma)$) if
for all $\gamma\in\Gamma$, $\gamma_*(\mu)$ and $\mu$ are absolutely continuous, and
$$\frac{d\gamma_*(\mu)}{d\mu}(z)=e^{-\delta(\Gamma)B(\gamma^{-1},z)}$$
for almost every $z\in\Lambda(\Gamma)$. This quasi-invariance property says that $\mu$ transforms like $\delta(\Gamma)$-dimensional Hausdorff measure.

\begin{theorem}[Patterson \cite{patterson}]
\label{psexists}
If $\Gamma$ is a  non-elementary Fuchsian group, then it admits a Patterson--Sullivan measure.
\end{theorem}

\begin{proof}[Sketch of proof:]
We will first give the proof in the simpler case when $P_\Gamma(\delta(\Gamma))=+\infty$.
In fact, this is always true
when $\Gamma$ is finitely generated. (However, there are finitely generated 
subgroups of $\mathrm{Isom}^+(\mathbb H^3)$ whose Poincare series converges at its critical exponent.)

For all $s>\delta(\Gamma)$ we construct the probability measure
$$\mu_s=\frac{1}{P_\Gamma(s)}\sum_{\gamma\in\Gamma} e^{-sd(i,\gamma(i))}\mathcal D_{\gamma(i)}$$
where $\mathcal D_{\gamma(i)}$ is the probability measure with support $\{\gamma(i)\}$. We may
regard $\mu_s$ as a probability measure on the compact space $\overline{\Gamma(i)}=\Gamma(i)\cap\Lambda(\Gamma)$.
So there exists a sequence
$s_n\to \delta(\Gamma)$ so that $s_n>\delta(\Gamma)$ for all $n$ and $\mu_{s_n}$ converges weakly to a 
measure $\mu$  on $\overline{\Gamma(i)}$.

For any $R>0$, 
$$\mu_{s_n}(B(i,R))\le \frac{1}{P_\Gamma(s_n)} \#\Gamma(R).$$
Since $P_\Gamma(s_n)\to \infty$, we see that $\mu_{s_n}(B(i,R))\to 0$, so
$\mu(B(i,R))=0$. It follows that the support of $\mu$ is contained in the limit set $\Lambda(\Gamma)$.

We now check the quasi-invariance property. If $\alpha\in\Gamma$ and $s>\delta(\Gamma)$, then
\begin{eqnarray*}
\alpha_*\mu_s & = & \sum_{\gamma\in\Gamma} e^{-sd(i,\gamma(i))}\alpha_*\mathcal D_{\gamma(i)}\\
& = & \sum_{\gamma\in\Gamma} e^{-sd(i,\gamma(i))}\mathcal D_{\alpha\gamma(i)}\\
& = & \sum_{\gamma\in\Gamma} e^{-sd(i,\alpha^{-1}\gamma(i))}\mathcal D_{\gamma(i)}\\
& = & \sum_{\gamma\in\Gamma} \left( e^{-s\big(d(i,\alpha^{-1}\gamma(i))-d(i,\gamma(i))\big)}\right) e^{-sd(i,\gamma(i))}\mathcal D_{\gamma(i)}\\
& = & \sum_{\gamma\in\Gamma} \left( e^{-s\big(d(\alpha(i),\gamma(i))-d(i,\gamma(i)\big)}\right) e^{-sd(i,\gamma(i))}\mathcal D_{\gamma(i)}\\
\end{eqnarray*}
for all $\alpha\in\Gamma$ and $s>\delta(\Gamma)$. Notice that if $\{\gamma_n(i)\}\to z\in\Lambda(\Gamma)$, then
$$\lim_{n\to\infty} d(\alpha(i),\gamma_n(i))-d(i,\gamma_n(i))=-b(z,\alpha(i))=B(\alpha^{-1},z)$$
so
$$\frac{d\alpha_*(\mu)}{d\mu}(z)=e^{-\delta(\Gamma)B(\alpha^{-1},z)}$$
for all $\alpha\in\Gamma$ and $z\in\Lambda(\Gamma)$. This finishes the sketch of the proof when $P_\Gamma(\delta(\Gamma))=+\infty$.

If $P_\Gamma(\delta(\Gamma))<+\infty$, one makes use of the following lemma:

\begin{lemma}[Patterson]
There exists a continuous non-decreasing function $h:\mathbb R_{\ge 0}\to\mathbb R_{\ge 0}$ such that the series
$$\hat P_\Gamma(s)=\sum_{\gamma\in\Gamma} h\left(e^{sd(i,\gamma(i))}\right) e^{-sd(i,\gamma(i))}$$
has critical exponent $\delta(\Gamma)$, $\hat P_{\Gamma}(\delta(\Gamma))=+\infty$, and
for any $\epsilon>0$, there exists $\Theta_\epsilon$ so that  if $t>1$ and $\theta>\Theta_\epsilon$, then $h(\theta t)\le t^\epsilon h(\theta)$.
\end{lemma}

One then defines 
$$ \mu_s=\frac{1}{\hat P_\Gamma(s)}\sum_{\gamma\in\Gamma}   h\left(e^{sd(i,\gamma(i))}\right) e^{-sd(i,\gamma(i))}\mathcal D_{\gamma(i)}$$
and proceeds much as above. 
\end{proof}

Given a Patterson--Sullivan measure $\mu$ for a non-elementary Fuchsian group $\Gamma$, one may then construct a measure on $T^1X_\Gamma$ which is supported on $T^1X_\Gamma^{nw}$.  We first define
the measure 
$$d\tilde m(w,z,t)=e^{\delta(\Gamma) G(w,z)}d\mu(w) d\mu(z) ds(t)$$
on $T^1X_\Gamma$, where $ds$ is the Lebesgue measure on $\Rb$. For the same reasons as the Lebesgue measure on $T^1X_\Gamma$, $\tilde m$ is $\Gamma$-invariant, $\varphi_s$-invariant, and is supported on $\Lambda(\Gamma)^{(2)}\times\mathbb R$.
Therefore, $\tilde m$ descends to a $\varphi_s$-invariant measure $m$ supported on $T^1X_\Gamma^{nw}$, which we call the \emph{Bowen--Margulis--Sullivan (BMS) measure}.

A key tool used to analyze the Patterson--Sullivan measure is Sullivan's Shadow lemma which provides estimates of the Patterson--Sullivan measure on intervals in $\partial\mathbb H^2$. For any $x,y\in\Hb^2$, the \emph{shadow} from $x$ of a ball of radius $r$ about a point $y$ is the set of endpoints of geodesic rays emanating
from $x$ which intersect the ball. More concretely,
$$\mathcal O_r(x,y)=\{z\in \partial\mathbb H^2 : d(\overline{xz},y)<r\}.$$
In the metric on $ \partial\mathbb H^2$ given by the identification $\partial\mathbb H^2\cong S^1$, the shadow $S_r(i,y)$ is an interval of radius roughly $e^{-d(i,y)}$.

\begin{theorem}[{Sullivan's Shadow Lemma \cite[Prop. 3]{sullivan-density}}]
Suppose $\Gamma$ is a non-elementary Fuchsian group and $\mu$ is a Patterson--Sullivan measure for $\Gamma$. There exists $R_0>0$ such that: if $r \geq R_0$, then there exists $C=C(r)$ so that
$$\frac{1}{C} e^{-\delta(\Gamma)d(i,\gamma(i))}\le \mu(S_r(i,\gamma(i))) \le C e^{-\delta(\Gamma)d(i,\gamma(i))}.$$
\end{theorem}

\begin{proof}[Sketch of proof] One first shows that there exists $\epsilon_0>0$ and $R_0>0$ so that
$$\mu(\mathcal O_{R_0}(\gamma(i),i))\ge \epsilon_0$$
for all $\gamma\in \Gamma$. If not, then there exists $\{\gamma_n\}\subset \Gamma$ so that
\[\lim_{n\to\infty}\mu(\mathcal O_n(\gamma_n(i),i))=0.\] 
One may then pass to a subsequence so that either $\mathcal O_n(\gamma_n(i),i))\to \partial \mathbb H^2$
(if $d(\gamma_n(i),i)<n$ for all $i$) or  $\mathcal O_n(\gamma_n(i),i))\to \partial \mathbb H^2-\{z\}$ (if $d(\gamma_n(i),i)\ge n$ for all $i$ and $\gamma_n(i)\to z$). We obtain
a contradiction, since by the quasi-invariance property and the assumption that $\Gamma$ is non-elementary, $\mu( \partial \mathbb H^2-\{z\} )$ is always non-zero.

Now, suppose $r \ge R_0$. The quasi-invariance property implies that for all $\gamma\in\Gamma$,
$$\mu(\mathcal O_r(\gamma^{-1}(i),i))=\gamma_*(\mu)(\mathcal O_r(i,\gamma(i)))=\int_{\mathcal O_r(i,\gamma(i))}e^{-\delta(\Gamma)B(\gamma^{-1},z)}\ d\mu.$$
One may then check that if  $z\in \mathcal O_r(i,\gamma(i))$, then
$$\big|d(i,\gamma(i))+B(\gamma^{-1},z)\big|\le 2r.$$
It then follows that
$$e^{-2r\delta(\Gamma)}e^{\delta(\Gamma)d(i,\gamma(i))} \mu(\mathcal O_r(i,\gamma(i)))\le \int_{\mathcal O_r(i,\gamma(i))}e^{-\delta(\Gamma)B(\gamma^{-1},z)}\ d\mu=\mu(\mathcal O_r(\gamma^{-1}(i),i))\le 1$$
and
$$\epsilon_0\le\mu(\mathcal O_r(\gamma^{-1}(i),i))=\int_{\mathcal O_r(i,\gamma(i))}e^{-\delta(\Gamma)B(\gamma^{-1},z)}\ d\mu\le e^{2r\delta(\Gamma)}e^{\delta(\Gamma)d(i,\gamma(i))} \mu(\mathcal O_r(i,\gamma(i))),$$
so we may conclude that the Shadow Lemma holds with $C(r)=\frac{e^{2r\delta(\Gamma)}}{\epsilon_0}$.
\end{proof}

We will consider two special types of  points in the limit set of a Fuchsian group $\Gamma$. A point $z\in\Lambda(\Gamma)$ is a \emph{conical limit point}, if there exists $r>0$ and
a sequence $\{\gamma_n\}\subset \Gamma$ so that $\gamma_n(i)\to z$ and $d(\gamma_n(i),\overline{iz})< r$ for all $n$.  
We let $\Lambda^{con}(\Gamma)$ denote the set of  all conical limit points.
A point $z\in\Lambda(\Gamma)$
is a \emph{bounded parabolic fixed point} if it is a fixed point of a parabolic element $\alpha\in\Gamma$ and $\langle \alpha\rangle$ acts cocompactly on
$\Lambda(\Gamma)-\{z\}$. (In the case of $\mathsf{PSL}(2,\mathbb R)$ every parabolic fixed point is bounded, but we give this more complicated 
definition to mimic the more general case.)

We say that a Fuchsian group $\Gamma$ is \emph{convex cocompact} if every point in $\Lambda(\Gamma)$ is conical. 
Geometrically, this is equivalent to $X_\Gamma$ containing a (non-empty) convex, compact subset bounded by finitely many closed geodesics.
We say that $\Gamma$ is \emph{geometrically finite} if every point in $\Lambda(\Gamma)$ is conical or a bounded parabolic fixed point. 
Geometrically, it is equivalent to $X_\Gamma$ containing a convex, finite area submanifold bounded by finitely many closed geodesics.

Sullivan's Shadow Lemma has the following nearly immediate consequences:

\begin{corollary}\label{shadow cor}
If $\Gamma$ is a non-elementary Fuchsian group and $\mu$ is a Patterson--Sullivan measure for
$\Gamma$, then:
\begin{enumerate}
\item
If $z\in\Lambda^{con}(\Gamma)$, then $\mu(\{z\})=0$.
\item
If $P_\Gamma(\delta(\Gamma))<+\infty$, then $\mu(\Lambda^{con}(\Gamma))=0$.
\end{enumerate}
In particular, if $\Gamma$ is convex cocompact, then $P_\Gamma(\delta(\Gamma))=+\infty$ and $\mu$ has no atoms.
\end{corollary} 

\begin{proof}[Sketch of proof] 
(1) Since  $z$ is conical, there exists a sequence $\{\gamma_n\}$ in $\Gamma$ and $r>0$ so that $\gamma_n(i)\to z$ and 
$d(\gamma_n(i),\overline{iz})<r$. In particular,
$z\in \mathcal O_r(i,\gamma_n(i))$ for all $n$. We may assume that $r>R_0$, where $R_0$ is the constant from the Shadow Lemma. Since $d(i,\gamma_n(i))\to\infty$,
the Shadow Lemma implies that $\mu(\mathcal O_r(i,\gamma_n(i)))\to 0$, which implies that $\mu(\{z\})=0$.

(2) Suppose $P_\Gamma(\delta(\Gamma))<+\infty$.
For any $r>0$, we say that $z\in \Lambda^{con}(\Gamma)$ is \emph{$r$-conical} with respect to $i$ if there exists a sequence $\{\gamma_n\}$ in $\Gamma$ so that $\gamma_n(i)\to z$ and $d(\gamma_n(i),\overline{iz})<r$. 
Let $\Lambda^{con}_r(\Gamma)$ denote the set of $r$-conical limit points with respect to $i$. Notice that
$$\Lambda^{con}(\Gamma)=\bigcup_{r>0}  \Lambda^{con}_r(\Gamma).$$
is an increasing family of subsets, so  if suffices to prove that if $r>R_0$, then $\mu( \Lambda^{con}_r(\Gamma))=0.$

Fix an enumeration $\{\gamma_1,\gamma_2,\ldots\}$ of $\Gamma$.
If $z\in \Lambda^{con}_r(\Gamma)$, then $z\in \mathcal O_r(i,\gamma_m(i))$ for infinitely many values of $m$, so for any $N$
$$ \Lambda^{con}_r(\Gamma)\subset\bigcup_{n\ge N} \mathcal O_r(i,\gamma_n(i)).$$
Therefore, by the Shadow Lemma,
$$\mu(\Lambda^{con}_r(\Gamma))\le \sum_{n=N}^\infty \mu(\mathcal O_r(i,\gamma_n(i))\le C(r)\sum_{n=N}^\infty e^{-\delta(\Gamma)d(i,\gamma_n(i))}$$
for any $N$.
However, since $P_\Gamma(\delta(\Gamma))<+\infty$, it follows that
\[\lim_{N\to\infty} \sum_{n=N}^\infty e^{-\delta(\Gamma)d(i,\gamma_n(i))}= 0\quad\text{so}\quad  \mu(\Lambda^{con}_r(\Gamma))=0.\qedhere\]
\end{proof}

A simple but beautiful argument of Dal'bo--Otal--Peign\'e  \cite{DOP}, yields the following  (extension of a) theorem of Brooks:

\begin{theorem}[Brooks \cite{brooks}] \label{entropy drop}
Suppose $\Gamma$ is a non-elementary Fuchsian group.
If $\Gamma_0\subset \Gamma$ is a proper subgroup such that $\Lambda(\Gamma_0)\subset\Lambda(\Gamma)$ is a proper subset and $P_{\Gamma_0}(\delta(\Gamma_0))=+\infty$, then 
$$\delta(\Gamma_0)<\delta(\Gamma).$$
\end{theorem}

\begin{proof}[Sketch of proof]
Let $\mu$ be a Patterson--Sullivan measure for $\Gamma$.
Fix an open subset $W\subset\Lambda(\Gamma)$ such that $\overline W \cap \Lambda(\Gamma_0)=\emptyset$. Since $\Gamma_0$ acts properly
discontinuously on $\partial\mathbb H^2-\Lambda(\Gamma_0)$, we see that  
$$
N: = \#\{ \gamma \in \Gamma_0 : \gamma W \cap W \neq \emptyset\}
$$
is finite. 

Since $\Lambda(\Gamma_0)$ is the set of accumulation points of $\{ \gamma(i) : \gamma \in \Gamma_0\}$, 
there is a finite subset $F\subset \Gamma_0$ and $\epsilon>0$ so  that the  angle between hyperbolic geodesics
$\overline{i\gamma(i)}$ and $\overline{iz}$ is at least  $\epsilon $
for all  $\gamma\in\Gamma_0-F$ and $z\in W$. This implies that there exists $D >0$ so that 
$$|B(\gamma,z)-d(i,\gamma(i))|\le D$$
for all  $\gamma\in\Gamma_0-F$ and $z\in W$.

Since $\Gamma_0\subset\Gamma$, it is immediate that $\delta(\Gamma)\ge\delta(\Gamma_0)$. 
Suppose for contradiction that  $\delta:=\delta(\Gamma)=\delta(\Gamma_0)$. Notice that 
$$\mu(\gamma(W))=(\gamma^{-1}_*\mu)(W)=\int_W  e^{-\delta B(\gamma,z)} d\mu(z)
\ge e^{-\delta D}e^{-\delta d(i,\gamma(i))} \mu(W)$$
for all  $\gamma\in\Gamma_0-F$. Since $P_{\Gamma_0}(\delta)=+\infty$,
$$1=\mu(\Lambda(\Gamma))\ge\frac{1}{N}\sum_{\gamma\in \Gamma_0}\mu(\gamma(W))\ge 
\frac{e^{-\delta D}\mu(W)}{N}\sum_{\gamma\in \Gamma_0-F} e^{-\delta d(i,\gamma(i))}=+\infty$$
which is a contradiction.
\end{proof}

If $\Gamma$ is a non-elementary Fuchsian group, $z\in\Lambda(\Gamma)$ is bounded parabolic fixed point and $H=\mathrm{Stab}_\Gamma(z)$ is its stabilizer, then $H$ is conjugate to $\langle w\to w+1\rangle$ and one may easily check that $\delta(H)=\frac{1}{2}$ and $P_H(\frac{1}{2})=+\infty$. Then by Theorem \ref{entropy drop}, $\delta(H)<\delta(\Gamma)$. Using this, one can examine the construction of a   Patterson--Sullivan measure $\mu$ given in Theorem \ref{psexists} to show that 
$\mu(\{z\})=0$. If $\Gamma$ is geometrically finite, $\Lambda(\Gamma)$ has at most countably many parabolic fixed points, so $\mu(\Lambda^{con}(\Gamma))=1$. Therefore, Corollary \ref{shadow cor} implies that if $\Gamma$ is geometrically finite, then $P_\Gamma(\delta(\Gamma))=+\infty$.
This gives the following corollary of  Brooks' theorem:

\begin{corollary} If $\Gamma$ is a non-elementary Fuchsian group and 
$\Gamma_0\subset \Gamma$ is an infinite index, geometrically finite subgroup, then
$$\delta(\Gamma_0)<\delta(\Gamma).$$
\end{corollary}

Hopf's argument \cite{hopf} for the ergodicity of the geodesic flows of closed hyperbolic manifolds
involves an analysis of the stable and unstable manifolds of the geodesic flow. The stable and unstable sets for
the geodesic flow on $T^1X_\Gamma^{nw}$  are simply the restriction of the well-studied stable and unstable manifolds
for the geodesic flow on $T^1X_\Gamma$. Hopf's argument  generalizes to show that the geodesic flow on
$T^1X_\Gamma^{nw}$ is ergodic with 
respect to the BMS measure when $P_\Gamma(\delta(\Gamma))=+\infty$. More generally, one can prove 
the following version of the celebrated Hopf--Tsuji--Sullivan dichotomy for Fuchsian groups.

\begin{theorem}[see Sullivan \cite{sullivan-density}]
\label{HST}
Suppose $\Gamma\subset\mathsf{PSL}(2,\mathbb R)$ is a non-elementary Fuchsian group and $\delta : = \delta(\Gamma)$.
\begin{itemize}
\item If 
$P_\Gamma^\phi(\delta)=+\infty$, 
then there exists a unique Patterson--Sullivan measure $\mu$  for $\Gamma$. Moreover:
\begin{enumerate}
\item $\mu(\Lambda^{con}(\Gamma))=1$.
\item The action of the geodesic flow on $T^1X_\Gamma$ is ergodic with respect to the Bowen--Margulis--Sullivan measure $m$.
\item The action of $\Gamma$ on $(\Lambda(\Gamma)^2, \mu \otimes \mu)$ is ergodic. 
\item The action of $\Gamma$ on $(\Lambda(\Gamma), \mu)$  is ergodic. 
\item The flow $\varphi_s$ on $T^1X_\Gamma^{nw}$ is ergodic and conservative.
\end{enumerate} 
\item If $P_\Gamma^\phi(\delta) < +\infty$ and  $\mu$ is any Patterson--Sullivan measure, then 
\begin{enumerate}
\item $\mu(\Lambda^{con}(\Gamma))=0$.
\item The action of the geodesic flow on $T^1X_\Gamma$ is not ergodic with respect to the Bowen--Margulis--Sullivan measure associated to $\mu$. 
\item The action of $\Gamma$ on $(\Lambda(\Gamma)^2, \mu \otimes \mu)$ is not ergodic. 
\item The flow $\varphi_s$ on $T^1X_\Gamma^{nw}$ is non-ergodic and dissipative.
\end{enumerate} 
\end{itemize}
\end{theorem}

We may regard $\mathbb H^d=\mathsf{SO}_0(d,1)/\mathsf{SO}(d)$, in which case $\mathsf{SO}_0(d,1)$ is  the group of orientation preserving isometries of $\Hb^d$.
We may choose the basepoint $x_0=[\mathsf{SO}(d)]\in\mathbb H^d$ and
identify $\partial \mathbb H^d$ with $T^1_{x_0}(\mathbb H^d)$ and hence with $S^{d-1}$. With these conventions all
the definitions and results go through for discrete  subgroups of $\mathsf{SO}_0(d,1)$ which are not virtually abelian,
where $x_0$ plays the role of $i$.

However, one key difference is that  any discrete,  finitely generated subgroup  of  $\mathsf{SO}_0(2,1)\cong \mathsf{PSL}(2,\mathbb R)$ is geometrically finite
and hence its Poincar\'e series diverges at its critical exponent. However, if $\Gamma$ is any discrete, finitely generated subgroup
of $\mathsf{SO}_0(3,1)$ which is not geometrically finite and whose limit set is not the entire sphere, then its critical exponent is two and its Poincar\'e series converges
at its critical exponent (see \cite{canary-laplace} and \cite{canary-survey}). Work of Greenberg \cite{greenberg}  shows that such groups are plentiful in boundaries of 
quasiconformal deformation spaces of geometrically finite subgroups of $\mathsf{SO}_0(3,1)$. 
If  $X=\mathbb H^3/\Gamma$ is a hyperbolic 3-manifold which fibers over the circle and $\Gamma_0$ is the subgroup of $\Gamma$ associated to a fibre,
then $\delta(\Gamma_0)=\delta(\Gamma)=2$ (see \cite{canary-laplace}), so the assumption in Brooks' theorem that $\Lambda(\Gamma_0)$ is a proper subset of
$\Lambda(\Gamma)$ is essential.

Roblin \cite{roblin} used techniques developed in Margulis' thesis \cite{margulis} to obtain counting results in this setting.

\begin{theorem}[Roblin \cite{roblin}] Let $\Gamma$ be a geometrically finite, torsion-free, but not virtually abelian, subgroup 
of $\mathsf{SO}_0(d,1)$, let $X:=\mathbb H^d/\Gamma$, and let $N_X(T)$ denote the number of closed geodesics on $X$ of length at most $T$. Then
$$N_X(T)\sim \frac{e^{\delta(\Gamma)T}}{\delta(\Gamma) T}.$$
\end{theorem}

Sullivan \cite{sullivan-density} used his Shadow Lemma to investigate the Hausdorff dimension of limit sets of geometrically finite groups.

\begin{theorem}[Sullivan \cite{sullivan-hd,sullivan-density}] If $\Gamma$ is a geometrically finite subgroup of $\mathsf{SO}_0(d,1)$ (which is
 not virtually abelian), then $\delta(\Gamma)$ is the Hausdorff dimension of $\Lambda(\Gamma)$. Moreover, if $\Gamma$ is convex cocompact, then the
Patterson--Sullivan measure is (a scalar multiple of) $\delta(\Gamma)$-dimensional Hausdorff measure. 
\end{theorem}

In general, $\delta(\Gamma)$ need not be the Hausdorff dimension of the limit set (see Patterson \cite{patterson-examples}), 
but Bishop and Jones showed that it is  always the Hausdorff dimension of the conical limit set.

\begin{theorem}[Bishop--Jones \cite{bishop-jones}] If $\Gamma\subset\mathsf{SO}_0(d-1,1)$ is discrete, then $\delta(\Gamma)$ is
the Hausdorff dimension of the conical limit  set $\Lambda^{con}(\Gamma)$.
\end{theorem}

\medskip\noindent
{\bf Remark:} 
If  $\Gamma$ is a finitely generated, discrete subgroup  of $ \mathsf{SO}_0(3,1)$, then $\delta(\Gamma)$ is the Hausdorff dimension of
its limit set.  Moreover,  either $\Lambda(\Gamma)$
has Lebesgue measure zero or $\Lambda(\Gamma)=\partial\mathbb H^3$ and $\Gamma$ acts ergodically on $\Lambda(\Gamma)$.
(This result was known as the Ahlfors Measure Conjecture.) See \cite{canary-survey} for a more complete discussion of the 3-dimensional case with full references.

\section{Divergent and transverse groups}

While the symmetric space
$X_d=\mathsf{SL}(d,\mathbb R)/\mathsf{SO}(d)$ associated to $\mathsf{SL}(d,\mathbb R)$ has non-positive curvature, it is not negatively curved because it  contains totally geodesic embedded copies of the Euclidean $(d-1)$-plane. As a result, the discrete subgroups of $\mathsf{SL}(d,\mathbb R)$ can exhibit much wilder behavior as compared to the discrete subgroups of $\PSL(2,\Rb)$ (or even the discrete subgroups of rank $1$ Lie groups). Thus, in order to find a generalization of the Patterson--Sullivan theory described in the previous section to the higher rank setting, we restrict ourselves to a class of discrete subgroups of $\SL(d,\Rb)$, called transverse subgroups. We will now discuss the theory of these subgroups, with a focus on how they act on their limit sets. 

We first recall some standard notions from the Lie theory of $\SL(d,\Rb)$. Let $\mathfrak{a}$ denote the vector space of $d\times d$, real valued, trace-free diagonal matrices and let $\mathfrak{a}^+$ denote the set of matrices in $\mathfrak{a}$ whose diagonal entries are weakly decreasing down the diagonal. We refer to $\mathfrak{a}$ as the \emph{(standard) Cartan subspace} of $\mathfrak{sl}(d,\mathbb R)$ and $\mathfrak{a}^+$ as the \emph{(standard) positive Weyl chamber}. It will be convenient to  identify $\mathfrak{a}$ and
$\mathfrak{a}^+$ with subsets of $\mathbb R^d$ by simply recording the diagonal. For all $k\in\{1,\dots,d-1\}$, the \emph{$k$-th simple root} of $\mathsf{SL}(d,\mathbb R)$ is the element $\alpha_k\in\mathfrak{a}^*$ given by
$$
\alpha_k(\vec a)=a_k-a_{k+1}\quad\text{for all}\quad \vec a=(a_1,\dots,a_d)\in\mathfrak a.
$$
Notice that $\vec a\in\mathfrak{a}^+$ if and only if $\alpha_k(\vec a)\ge 0$ for all $k$, and that $\{\alpha_1,\ldots,\alpha_{d-1}\}$ is a basis of $\mathfrak{a}^*$. 

The \emph{Cartan projection} $\kappa:\mathsf{SL}(d,\mathbb R)\to \mathfrak{a}^+$ is the map defined by
$$\kappa(A)=\big(\log\sigma_1(A),\ldots,\log\sigma_d(A)\big)$$
where $\sigma_1(A)\ge\dots\ge\sigma_d(A)$ are singular values of $A$. The vector $\kappa(A)$ can be interpreted as a $\mathfrak a^+$-valued version of how far $A\in\SL(d,\Rb)$ translates the identity coset in $X_d$, which we regard as a basepoint $x_0\in X_d$. Indeed, there is some constant $C>0$ such that for all $A\in\SL(d,\Rb)$, $d_{X_d}(x_0,A(x_0))=C||\kappa(A)||$. 

The Cartan decomposition theorem says that for all $A\in\mathsf{SL}(d,\mathbb R)$ there exists $L,M\in \mathsf{SO}(d)$ then
$A=Le^{\kappa(A)}M$. Furthermore, even though $L$ (and $M$) is not unique to $A$, if 
\[\alpha_k(\kappa(A))=\log\left(\frac{\sigma_k(A)}{\sigma_{k+1}(A)}\right)>0,\] 
then
\[U_k(A):=L\Big({\rm Span}_{\Rb}(e_1,\dots,e_k)\Big)\]
is well-defined, i.e. does not depend on the choice of $L$. Geometrically, $\sigma_k(A)$ is the length of the $k$-th largest semi-axis of the ellipsoid $A(\mathbb S^{d-1})$, and $U_k(A)$ is the span of its $k$ largest axes.

Let $\theta$ be a non-empty subset of $\{1,\ldots,d-1\}$ that is \emph{symmetric}, i.e. $k\in\theta$ if and only if $d-k\in\theta$.
We say that a subgroup $\Gamma\subset\mathsf{SL}(d,\mathbb R)$ is \emph{$\Psf_\theta$-divergent} if $\alpha_k(\kappa(\gamma_n))\to\infty$ for every infinite sequence $\{\gamma_n\}$ in $\Gamma$ and all $k\in\theta$. Notice that $\Psf_\theta$-divergent groups are discrete. An important property of $\Psf_\theta$-divergent groups is that they admit a natural notion of a limit set, which we will now construct.

Let $\mathcal F_\theta$ denote the $\theta$-flag manifold, i.e if $\theta=\{k_1<\cdots<k_m\}$, then $\mathcal F_\theta$ is the set of nested tuples of subspaces 
$F=(F^{k_1}, F^{k_2},\dots, F^{k_m})$ where $F^{k_i}\subset\Rb^d$ is a vector subspace of dimension $k_i$. For any $A\in\mathsf{SL}(d,\mathbb R)$ such that $\alpha_k(\kappa(A))>0$ for all $k\in\theta$, denote
\[U_\theta(A):=(U_{k_1}(A),\dots,U_{k_m}(A))\in\Fc_\theta.\]
If $\Gamma\subset\SL(d,\Rb)$ is $\Psf_\theta$-divergent, then $U_\theta(\gamma)$ is defined for all but finitely many $\gamma\in\Gamma$. We then define
the \emph{$\theta$-limit set of $\Gamma$}, denoted $\Lambda_\theta(\Gamma)$, to be the set of accumulation points of $\{U_\theta(\gamma):\gamma\in\Gamma\}$ in $\Fc_\theta$, i.e. $F\in\Lambda_\theta(\Gamma)$ if and only if there 
exists a sequence of distinct elements $\{\gamma_n\}$ in $\Gamma$ so that $U_\theta(\gamma_n)\to F$. One can verify that $\Lambda_\theta(\Gamma)$ is a closed, $\Gamma$-invariant subset of $\Fc_\theta$. We say that $\Gamma$ is \emph{non-elementary} if $\Lambda_\theta(\Gamma)$ is infinite. 

Since $\theta$ is symmetric, we may say that two flags $F,G\in\Fc_\theta$ are \emph{transverse} if $F^k+G^{d-k}=\Rb^d$ for all $k\in\theta$. We then say that a $\Psf_\theta$-divergent group is \emph{$\Psf_\theta$-transverse} if its $\theta$-limit set consists of mutually transverse flags. 

Recall that the action, by homeomorphisms, of a  group $\Gamma_0$ on a 
compact metrizable space $X$ is said to be a (discrete)  {\em convergence group action} if for every sequence of distinct elements $\{\gamma_n\}$ in $\Gamma_0$,  there are points $x,y\in X$ and a subsequence, still denoted $\{\gamma_n\}$, so that $\gamma_n(z)$ converges to  $x$ for all $z\in X\setminus\{y\}$ (uniformly on compact subsets of $X\setminus\{y\}$). It is a consequence of the Cartan decomposition theorem (see \cite[Proposition 2.3]{CZZ3}) that if $\Gamma$ is $\Psf_\theta$-divergent and $\{\gamma_n\}$ is a sequence of distinct elements in $\Gamma$ such that $U_\theta(\gamma_n)\to F^+$ and $U_\theta(\gamma_n^{-1})\to F^-$ for some $F^\pm\in\Lambda_\theta(\Gamma)$, then $\gamma_n(F)\to F^+$ for all $F\in\Fc_\theta$ transverse to $F^-$. 
This allows us to deduce the following crucial feature of $\Psf_\theta$-transverse groups. 

\begin{proposition}[{\cite[Section 5.1]{KLP}, \cite[Proposition 3.3]{CZZ2}}]\label{prop: convergence group}
If $\Gamma$ is $\Psf_\theta$-transverse, then $\Gamma$ acts on $\Lambda_\theta(\Gamma)$ as a convergence group. 
In particular, if $\Gamma$ is non-elementary, then $\Gamma$ acts on $\Lambda_\theta(\Gamma)$ minimally, and $\Lambda_\theta(\Gamma)$ is perfect. 
\end{proposition}

For subgroups of $\SL(d,\Rb)$ that lie in $\mathsf{SO}_0(d-1,1)$, being discrete is equivalent to being $P_{1,d-1}$-divergent, which is in turn equivalent to being $\Psf_{1,d-1}$-transverse. The first equivalence is a consequence of the observation that if $A\in\mathsf{SO}_0(d-1,1)$, then
$$d_{\mathbb H^{d-1}}(x_0,A(x_0))=\alpha_1(\kappa(A))=\alpha_{d-1}(\kappa(A))$$ 
where $x_0=[\mathsf{SO}(d-1)]\in\mathbb H^{d-1}=\mathsf{SO}_0(d-1,1)/\mathsf{SO}(d-1)$. (Notice that if $k\ne 1,d-1$, then $\alpha_k(A)=0$, so $\Gamma$ is not $\Psf_{k,d-k}$-divergent.) For the second equivalence, note that the inclusion of $\mathsf{SO}_0(d-1,1)\subset\SL_d(\Rb)$ induces an identification between $\partial\Hb^{d-1}$ and  the projectivization of the light cone in $\Rb^d$. Furthermore, if $L$ is a line in the light cone, then its perpendicular $L^\perp$ is the hyperplane that is tangent to the light cone and contains $L$. Thus, if $\Gamma\subset\mathsf{SO}_0(d-1,1)$ is discrete and $\Lambda(\Gamma)\subset\partial\mathbb H^{d-1}$ denotes its traditional limit set, then
$$\Lambda_{1,d-1}(\Gamma)=\{ (L,L^\perp)\in\Fc_{1,d-1}:L\in\Lambda(\Gamma)\},$$
which is clearly a set of mutually transverse partial flags. More generally, for rank $1$ semisimple Lie groups, the notions of discrete subgroups, divergent subgroups, and transverse subgroups agree.

In $\SL(d,\Rb)$ (and more generally, in higher rank semisimple Lie groups)  well-studied examples of transverse groups are Anosov and relatively Anosov groups. To describe these, we recall some terminology from the theory of convergence groups actions. Suppose that a group $\Gamma_0$ acts on a copact metrizable space $X$ as a convergence group. We say that a point $x\in X$  is a  
{\em conical limit point} if there exist distinct $a,b\in X$ and a sequence
$\{\gamma_n\}$ in $\Gamma_0$ so that $\gamma_n(x)$ converges to $a$ and $\gamma_n(y)$ converges to $b$ for all $y\in X\setminus\{x\}$. Also, we say that a point $x\in X$ is a \emph{bounded parabolic fixed point} if its stabilizer in $\Gamma$ acts properly and cocompactly on $X\setminus\{x\}$. These definitions are equivalent to the definitions already given in the case of limit sets of discrete subgroups of $\mathsf{SO}_0(d-1,1)$. We then say that a $\Psf_\theta$-transverse subgroup $\Gamma$ is \emph{$\Psf_\theta$-Anosov} if every point in $\Lambda_\theta(\Gamma)$ is a conical limit point of the $\Gamma$-action. Analogously, a $\Psf_\theta$-transverse subgroup $\Gamma$ is \emph{$\Psf_\theta$-relative Anosov} if every point in $\Lambda_\theta(\Gamma)$ is either a conical limit point or a  bounded parabolic fixed point. 

Notice from these definitions that $\Psf_\theta$-Anosov and $\Psf_\theta$-relative Anosov subgroups are respectively higher rank generalizations of convex cocompact and  geometrically finite subgroups of $\mathsf{SO}_0(d-1,1)$. Also, it follows from Bowditch \cite{Bowditch99} and Tukia \cite{tukia-conical} that every $\Psf_\theta$-Anosov subgroup is a hyperbolic group whose Gromov boundary is its $\theta$-limit set. Similarly, by Yaman \cite{yaman}, every $\Psf_\theta$-relative Anosov subgroup is a relative hyperbolic group (the peripheral groups are the stabilizers of the parabolic fixed points) whose Bowditch boundary is its $\theta$-limit set. 

The notion of Anosov subgroups of hyperbolic groups was introduced by Labourie \cite{labourie-invent}, and was further developed by
Guichard--Wienhard \cite{guichard-wienhard}, Kapovich--Leeb--Porti \cite{KLP}, Gu\'eritaud--Guichard--Kassel--Wienhard \cite{GGKW},
Bochi--Potrie--Sambarino \cite{BPS}
and others. This was later generalized to the notion of relatively Anosov subgroups by Kapovich--Leeb \cite{KL}, which was further developed by Zhu \cite{zhu-reldom}, Zhu--Zimmer \cite{ZZ1,ZZ2} and others. Both of these classes of representations are central in several areas of research, most notably in higher Teichm\"uller theory and convex real projective geometry.

\section{Patterson--Sullivan measures for divergent groups}

We will now introduce Patterson--Sullivan theory in the setting of divergent groups. In the $\mathsf{SO}_0(d-1,1)$ setting, three ingredients that were crucial to the Patterson-Sullivan theory are the translation distance of the base point $d(x_0,\bullet(x_0)):\mathsf{SO}_0(d-1,1)\to\Rb$, the Busemann cocycle $B:\mathsf{SO}_0(d-1,1)\times\partial\Hb^d\to\Rb$, and the length function $\ell:\mathsf{SO}_0(d-1,1)\to\Rb$. We will now define analogs of all three of these for $\SL(d,\Rb)$.

Let $\theta\subset\{1,\dots,d-1\}$ be non-empty and symmetric, and denote
\[\mathfrak a_\theta:=\{\vec a\in\mathfrak{a}:\alpha_k(\vec a)=0\ \text{if}\ k\notin\theta\}.\]
For any $k\in\{1,\dots,d-1\}$, the \emph{$k$-th fundamental weight} of $\mathsf{SL}(d,\mathbb R)$ is the element $\omega_k\in\mathfrak{a}^*$
given by
$$\omega_k(\vec a)=a_1+\cdots+a_k\quad\text{for all}\quad\vec a=(a_1,\dots,a_d)\in\mathfrak a.$$
Notice that $\{\omega_k|_{\mathfrak a_\theta}:k\in\theta\}$ is a basis for $\mathfrak{a}_\theta^*$, so there is a projection 
\[p_\theta:\mathfrak a\to\mathfrak a_\theta\] 
defined by $\omega_k(\vec a)=\omega_k(p_\theta(\vec a))$ for all $k\in\theta$. Then define $\kappa_\theta:=p_\theta\circ\kappa$. 

Now, suppose that $\Gamma\subset\SL(d,\Rb)$ is a $\Psf_\theta$-divergent subgroup. The \emph{$\theta$-Benoist limit cone of $\Gamma$}, denoted $\mathcal B_\theta(\Gamma)$, is the set of accumulation points of the lines $\{\mathbb R^+\kappa_\theta(\gamma):\gamma\in\Gamma\}$. It is known that $\Bc_\theta(\Gamma)$ is a closed pointed convex cone in $\mathfrak a_\theta$. Furthermore, $\phi\in\mathfrak a_\theta^*$ is non-negative on $\mathcal B_\theta(\Gamma)$ if and only if $\phi(\kappa_\theta(\gamma))\ge 0$ for all but finitely many $\gamma\in\Gamma$.
We refer the reader to the excellent book of Benoist and Quint \cite{BQ} for more details on these definitions.

Recall that if $\Gamma\subset\mathsf{SO}_0(d-1,1)$ is discrete and infinite, then $\Gamma$ is $\Psf_{1,d-1}$-transverse. Notice that even though $\mathfrak a_{1,d-1}$ is $2$-dimensional, 
\[\mathfrak a':=\kappa(\mathsf{SO}_0(d-1,1))=\kappa_{1,d-1}(\mathsf{SO}_0(d-1,1))\subset\mathfrak a_{1,d-1}\]
is $1$-dimensional, and $\Bc_{1,d-1}(\Gamma)=\{(k,0,\dots,0,-k):k\in\Rb^+\}$ is a ray in $\mathfrak a'$ (that does not depend on $\Gamma$). In other words, for any non-zero $\phi\in\mathfrak a_{1,d-1}^*$ that is non-negative on $\Bc_{1,d-1}(\Gamma)$, the function $\phi\circ\kappa_{1,d-1}|_{\mathsf{SO}_0(d-1,1)}$ is a positive scalar multiple of $\omega_1\circ\kappa_{1,d-1}|_{\mathsf{SO}_0(d-1,1)}=d(x_0,\bullet(x_0))$. 

By the observation above, for a general $\Psf_\theta$-transverse subgroup $\Gamma\subset\SL(d,\Rb)$ and general $\phi\in\mathfrak a_\theta^*$ that is non-negative on $\Bc_\theta(\Gamma)$, we regard
\[\phi\circ\kappa_\theta:\SL(d,\Rb)\to\Rb\]
as the required analog for the translation distance of the basepoint. Unlike the special case when $\Gamma\subset\mathsf{SO}_0(d-1,1)$, in general, the span of $\kappa_\theta(\Gamma)$ might be more than $1$-dimensional and $\Bc_\theta(\Gamma)$ will depend on $\Gamma$. Thus, the analog of the translation distance of the basepoint will depend on the choice of $\phi$, and the set of possible $\phi$ will depend on $\Gamma$.

To define the analog of the Busemann cocycle, we first define the \emph{partial Iwasawa cocycle}
$$B_\theta:\mathsf{SL}(d,\mathbb R)\times\mathcal F_\theta\to\mathfrak{a}_\theta,$$
to be the map such that for any $k\in\theta$,
$$\omega_k(B_\theta(A,F))=\log\frac{\norm{A(v_1)\wedge \cdots\wedge A(v_k)}_k}{\norm{v_1\wedge \cdots\wedge v_k}_k}$$
where $v_1,\dots,v_k$ is some (any) basis of $F^k$ and $\norm{\cdot}_k$ is the norm on $\bigwedge^k\Rb^d$ induced by the standard inner product on $\Rb^d$. Then for any $\phi\in\mathfrak a_\theta^*$ that is non-negative of $\Bc_\theta(\Gamma)$, define
\[\phi\circ B_\theta:\mathsf{SL}(d,\mathbb R)\times\mathcal F_\theta\to\Rb.\]
Note that the traditional Busemann cocycle for $\mathsf{SO}_0(d-1,1)$ is $\omega_1\circ B_{1,d-1}$, so we may regard $\phi\circ B_\theta$ as the required analog of the Busemann cocycle.
 
Finally, we define the analog of the length function. To do so, define the \emph{Jordan projection} $\nu:\SL(d,\Rb)\to\mathfrak a^+$ to be the map defined by 
$$\nu(A)=\big(\log\lambda_1(A),\ldots,\log\lambda_d(A)\big)$$
where $\lambda_1(A)\ge\dots\ge\lambda_d(A)$ are the modulus of the (generalized) eigenvalues of $A$. The vector $\nu(A)$ can be regarded as a $\mathfrak a^+$-valued version of the minimum translation distance of $A\in\SL_d(\Rb)$ on $X_d$. Indeed, there is some constant $C>0$ such that for all $A\in\SL(d,\Rb)$, $\inf_{x\in X_d}d_{X_d}(x,A(x))=C||\nu(A)||$. Let $\nu_\theta:=p_\theta\circ\nu$. Then for every $\phi\in\mathfrak a_\theta^*$ that is non-negative on $\Bc_\theta(\Gamma)$, define
\[\phi\circ\nu_\theta:\SL(d,\Rb)\to\Rb.\]
 (The assumption that $\phi$ is non-negative on $\Bc_\theta(\Gamma)$ is equivalent to assuming that $\phi(\nu_\theta(\gamma))\ge 0$ for all $\gamma\in\Gamma$.) Note that the traditional length function for $\mathsf{SO}_0(d-1,1)$ is $\omega_1\circ \nu_{1,d-1}$, so we may regard $\phi\circ\nu_\theta$ as the required analog of the length function.

We now have all the ingredients needed to discuss the Patterson--Sullivan theory for a $\Psf_\theta$-divergent subgroup $\Gamma\subset\SL(d,\Rb)$. For any $\phi\in\mathfrak a_\theta^*$ that is non-negative on $\Bc_\theta(\Gamma)$, the \emph{$\phi$-Poincar\'e series} of $\Gamma$ is
$$P_\Gamma^\phi(s):=\sum_{\gamma\in\Gamma}e^{-s\phi(\kappa_\theta(\gamma))},$$
and the \emph{$\phi$-critical exponent} $\delta^\phi(\Gamma)\subset [0,+\infty]$ of $\Gamma$ is the critical exponent of $P_\Gamma^\phi$. If $\phi$ is strictly positive on $\Bc_\theta(\Gamma)-\{\vec 0\}$, then $\delta^\phi(\Gamma)<+\infty$ (see \cite[Proposition 2.7]{CZZ3}). Moreover, the converse is true if $\Gamma$ is Anosov \cite[Corollary 5.7.2]{sambarino-dichotomy} or relatively Anosov \cite[Theorem 1.3]{CZZ4}. In the case when $\Gamma\subset\mathsf{SO}_0(d-1,1)$, then $\theta=\{1,d-1\}$ and $P^{\omega_1}_\Gamma$ is the traditional Poincar\'e series.

We say that a probability measure $\mu$ on $\Lambda_\theta(\Gamma)$ is a \emph{$\phi$-Patterson--Sullivan measure} for $\Gamma$ (of dimension $\delta^\phi(\Gamma)$) if for all $\gamma\in\Gamma$, $\gamma_*(\mu)$ and $\mu$ are absolutely continuous, and 
$$\frac{d\gamma_*(\mu)}{d\mu}(z)=e^{-\delta^\phi(\Gamma)\phi(B_\theta(\gamma^{-1},z))}$$
for almost every $z\in\Lambda_\theta(\Gamma)$. 

Albuquerque \cite{albuquerque} and Quint \cite{quint-ps} were the first to construct Patterson--Sullivan measures for Zariski dense, discrete subgroups 
of $\mathsf{SL}(d,\mathbb R)$ on the full flag manifold.  We prove the analogous result for divergent subgroups on the partial flag manifold. Also see related work
of Sambarino \cite{sambarino15} and Dey--Kapovich \cite{DK1}.

\begin{theorem}[{\cite[Prop. 3.2]{CZZ3}}]
\label{trans:psexists}
If $\Gamma\subset \mathsf{SL}(d,\mathbb R)$ is $\Psf_\theta$-divergent,
$\phi\in \mathfrak{a}_\theta^*$ and $\delta^\phi(\Gamma)<+\infty$, then $\Lambda_\theta(\Gamma)$ admits a Patterson-Sullivan measure.
\end{theorem}

\begin{proof}[Sketch of proof]
One may define a compact topology on $\Gamma\sqcup\Lambda_\theta(\Gamma)$ which agrees with the topologies on $\Gamma$ and
$\Lambda_\theta(\Gamma)$, and so that $\{\gamma_n\}\subset \Gamma$ converges to $z\in\Lambda_\theta(\Gamma)$ if and only
if $\{U_\theta(\gamma_n)\}$ converges to $z$. We then define, for all $s>\delta^\phi(\Gamma)$, the measure
$$\mu_s=\frac{1}{P_\Gamma^\phi(s)}\sum_{\gamma\in\Gamma} e^{-s\phi(\kappa_\theta(\gamma))} \mathcal D_\gamma.$$
One may then proceed exactly as in the proof of Theorem \ref{psexists} using the properties of the partial Iwasawa cocyle in place of
the properties of the Busemann cocycle.
\end{proof}

\section{The Shadow Lemma for transverse groups}

From our description of the $\PSL(2,\Rb)$ case, we see that to understand the ergodicity properties of Patterson--Sullivan measures of divergent subgroups in $\SL(d,\Rb)$, we need two more ingredients. The first is a notion of shadows in the limit set of the transverse group so that the analog of Sullivan's Shadow Lemma holds. The second is a flow space for the divergent subgroup that admits a Hopf decomposition. Unfortunately, neither of these are immediately available to us in this generality. However, if we further assume that our divergent group is transverse, then we have at our disposal a tool from projective geometry which allows us to access both of these. Later work of Kim--Oh--Wang \cite{KOW} and Blayac--Canary--Zhu--Zimmer \cite{BCZZ1,BCZZ2} give a way to bypass this tool, but we still present it here as we believe it helps to provide intuition for the generalized arguments given in the later works. 

In this section, we discuss this projective geometry tool, and how one can use it to define the required notion of shadows. We will also discuss the consequences of the analog of Sullivan's Shadow Lemma. 

We say that an open set $\Omega\subset \Pb(\Rb^{d_0})$ is a \emph{properly convex domain} if it is bounded and convex in some affine chart for
$\Pb(\Rb^{d_0})$. If $x,y\in\overline\Omega$, we let $\overline{xy}\subset\overline\Omega$ denote the closed line segment joining them in the affine chart. On $\Omega$, there is a natural Finsler (but typically non-Riemannian) metric on $\Omega$ called the \emph{Hilbert metric}, which is defined as followed. If $x,y\in\Omega$, then let $w,z$ be the points of intersection of the line through $x$ and $y$ with $\partial\Omega$,
labelled so that the points appear in the order $w,x,y,z$ on the line. We then define
$$d_\Omega(x,y)=\frac{1}{2} \log\left(\frac{\norm{w-y}\cdot \norm{x-z}}{\norm{w-x}\cdot \norm{y-z}}\right),$$
where $\norm{a-b}$ denotes some (any) Euclidean distance between $a$ and $b$ in some (any) affine chart that contains $\overline{\Omega}$. Projective line segments are geodesics for this metric, but they need not be the only geodesics. The Hilbert metric is invariant under
the group $\mathrm{Aut}(\Omega)$ of 
projective automorphisms of $\mathbb{RP}^{d_0}$ which preserve $\Omega$.
If $\Omega$ is an ellipsoid, then $(\Omega,d_\Omega)$ is isometric to $\mathbb H^{d_0-1}$ (in fact, up to translation by an element of $\PGL(d_0,\Rb)$ this is the only properly convex domain for which the Hilbert metric is induced by a Riemannian metric). 

If we pick a basepoint $b_0\in\Omega$, we may define shadows in analogy with the hyperbolic setting: If $p\in\Omega$ and
$r>0$, then the \emph{shadow} from $b_0$ of a ball of radius $r$ centered about $p$ is
$$\mathcal O_r(b_0,p)=\{z\in\partial\Omega : d_\Omega(p,[b_0,z])<r\}$$
where $[b_0,z]$ is the projective line segment in $\overline{\Omega}$ joining $b_0$ to $z$.

Indeed, in the case when $\Omega$ is an ellipsoid, these shadows 
agree with the ones defined for $\Hb^{d_0-1}$.
 
If $\Gamma_0\subset\mathsf{PGL}(d_0,\mathbb R)$ is a discrete subgroup that preserves a properly convex domain $\Omega\subset\Pb(\Rb^{d_0})$, we define its \emph{full orbital limit set} to be
$$\Lambda_\Omega(\Gamma_0)=\left\{z\in \partial \Omega : z=\lim \gamma_n(p)\text{ for some } p\in\Omega\text{ and } \{\gamma_n\}\subset\Gamma\right\}.$$
We then say that $\Gamma_0$ is \emph{projectively visible} if
\begin{enumerate}
\item for all $x,y\in\Lambda_\Omega(\Gamma_0)$, the interior of $[x,y]$ lies
in $\Omega$, and
\item every point $x\in\Lambda_\Omega(\Gamma_0)$ is a \emph{$C^1$-smooth  point}, i.e. there is a unique projective hyperplane in $\mathbb{RP}^{d_0}$ that contains $x$ but does not intersect $\Omega$.
\end{enumerate}

These conditions are not enough to guarantee that the Hilbert metric on $\Omega$ is Gromov hyperbolic,  but it ensures sufficient hyperbolic-like behavior to allow us to mimic classical arguments from the hyperbolic setting. In particular, we may define shadows in $\Lambda_\Omega(\Gamma_0)$ to be the shadows defined above intersected with $\Lambda_\Omega(\Gamma_0)$. 

To relate projective visible groups with transverse subgroups, we use ideas from the work of Zimmer \cite{Zimmer} and Danciger--Gueritaud--Kassel \cite{DGK} to prove the following theorem, which says that every transverse group admits a ``good'' isomorphism to a projectively visible group. This is the tool that allows us to define the shadows and the flow space for transverse subgroups. 

\begin{theorem}[{\cite[Theorem 6.2]{CZZ3}}]
\label{projvisible}
If $\Gamma\subset \mathsf{SL}(d,\mathbb R)$ is $\Psf_\theta$-transverse, then there exists a projectively
visible group $\Gamma_0\subset\mathrm{Aut}(\Omega)\subset\mathsf{PGL}(d_0+1,\mathbb R)$, an
isomorphism $\rho:\Gamma_0\to\Gamma$ and a $\rho$-equivariant homeomorphism
$\xi:\Lambda_\Omega(\Gamma_0)\to\Lambda_\theta(\Gamma)$.
\end{theorem}

For the remainder of this section, let $\Gamma\subset\SL(d,\Rb)$ be a $\Psf_\theta$-transverse subgroup, and let $\Gamma_0$, $\Omega$, $\rho$ and $\xi$ be as given in Theorem \ref{projvisible}. If we fix a basepoint $b_0\in\Omega$, then for all $p\in\Omega$ and $r>0$, we may think of
\[\xi\big( \Oc_r(b_0,p) \cap \Lambda_\Omega(\Gamma_0) \big)\subset\Lambda_\theta(\Gamma)\]
as the analog of a shadow. We prove that this notion of shadows satisfy a version of Sullivan's Shadow Lemma.
(Other higher rank analogues of Sullivan's Shadow Lemma have been obtained by Albuquerque \cite{albuquerque}, Link \cite{link}, Quint \cite{quint-ps},
Dey--Kapovich \cite{DK1}, Lee--Oh \cite{LeeOh} and Sambarino \cite{sambarino-dichotomy}.)

 \begin{proposition}[{\cite[Prop. 7.1]{CZZ3}}]
 \label{prop:shadow estimates}
Suppose that $\mu$ is a $\phi$-Patterson--Sullivan measure for $\Gamma$. For any $b_0 \in \Omega$, there exists $R_0>0$ such that: if  $r \ge R_0$, then
there exists $C=C(b_0,r)> 1$ so that 
\begin{align*}
C^{-1} e^{-\delta^\phi(\Gamma) \phi(\kappa_\theta(\rho(\gamma)))} \le \mu\Big(\xi\big( \Oc_r(b_0,\gamma(b_0)) \cap \Lambda_\Omega(\Gamma_0) \big) \Big)  \le Ce^{-\delta^\phi(\Gamma) \phi(\kappa_\theta(\rho(\gamma)))}
\end{align*}
for all $\gamma \in \Gamma_0$.
\end{proposition} 

As in the Fuchsian case we can derive the following immediate consequences. Let $\Lambda^{\rm con}_\theta(\Gamma)$ denote its {\em $\theta$-conical limit set}, i.e. the set of conical limit points for the action of $\Gamma$ on $\Lambda_\theta(\Gamma)$.

\begin{corollary}
[{\cite[Prop. 8.1]{CZZ3}}]
Suppose that $\mu$ is a $\phi$-Patterson--Sullivan measure.
\begin{enumerate}
\item
If $z\in\Lambda^{con}_\theta(\Gamma)$, then $\mu(\{z\})=0$,
and
\item
If $P_\Gamma^\phi(\delta^\phi(\Gamma))<+\infty$, then $\mu(\Lambda^{con}_\theta(\Gamma))=0$.
\end{enumerate}
In particular, if $\Gamma$ is $P_\theta$-Anosov, then $P_\Gamma^\phi(\delta^\phi(\Gamma))=+\infty$.
\end{corollary} 

By applying the argument of Dal'bo--Otal--Peign\'e  \cite{DOP}, we can then generalize Brooks' result.

\begin{theorem}[{\cite[Thm. 4.1]{CZZ3}}] \label{thm: entropy drop}
 Suppose that $\Gamma\subset\mathsf{SL}(d,\mathbb R)$ is $P_\theta$-transverse and $\phi\in\mathfrak{a}_\theta^*$ satisfies $\delta^\phi(\Gamma)<+\infty$. If $\Gamma_0\subset \Gamma$ is a subgroup such that $\Lambda_\theta (\Gamma_0)$ is a proper subset of $\Lambda_\theta(\Gamma)$
and $P_{\Gamma_0}^\phi(\delta^\phi(\Gamma_0))=+\infty$, then
$$\delta^{\phi}(\Gamma_0)<\delta^{\phi}(\Gamma).$$
\end{theorem}

Coulon--Dal'bo--Sambusetti \cite{CDS} proved a version of Brook's result for CAT($-1$) spaces: they showed that if $\Gamma$ admits a cocompact, properly discontinuous action on a CAT($-1$)-space, then a subgroup of $\Gamma$ has strictly smaller critical exponent if and only if is is co-amenable. In this direction, the most general existing results are due to Coulon--Dougall--Schapira--Tapie \cite{CDST} for strongly positively recurrent actions on Gromov hyperbolic spaces. See \cite{CDST} for a more complete
history.

Recall that if $\Gamma$ is a finitely generated group, then a finitely generated subbgroup  $\Gamma_0$  of $\Gamma$ is  quasiconvex if the inclusion map from $\Gamma_0$ into $\Gamma$
is a quasi-isometric embedding. If $\Gamma$ is $P_\theta$-Anosov, then any quasiconvex subgroup is itself $P_\theta$-Anosov, see \cite[Lem. 2.3]{CLSS}. This gives the following new corollary.

\begin{corollary}[{\cite[Cor. 4.2]{CZZ3}}] If $\Gamma\subset\mathsf{SL}(d,\mathbb R)$ is $P_\theta$-Anosov, $\phi\in\mathfrak{a}_\theta^*$ satisfies
$\delta^\phi(\Gamma)<+\infty$ and $\Gamma_0\subset \Gamma$ is an infinite index quasiconvex subgroup, then
$$\delta^{\phi}(\Gamma_0)<\delta^{\phi}(\Gamma).$$
\end{corollary}

A crucial technical result in  \cite{CZZ4} guarantees that if $\Gamma$ is $P_\theta$-relatively Anosov, $z\in\Lambda(\Gamma)$ is a bounded parabolic fixed point,
$H$ is the stabilizer of $z$ in $\Gamma$, and $\phi\in\mathfrak{a}_\theta^*$ satisfies $\delta^\phi(\Gamma)<+\infty$, then
$P_H^\phi(\delta^\phi(H))=+\infty$. This allows us to apply Theorem \ref{thm: entropy drop} to deduce that
$\delta^\phi(H)<\delta^\phi(\Gamma)$. One may then, as in the hyperbolic setting, examine the construction of a Patterson--Sullivan
measure $\mu$  in Theorem \ref{trans:psexists} to show that $\mu(\{z\})=0$. Therefore, we see that:

\begin{theorem}[{\cite{CZZ4}}]\label{thm: rel Anosov divergent}
If $\Gamma\subset\mathsf{SL}(d,\mathbb R)$ is $P_\theta$-relatively Anosov and $\phi\in\mathfrak{a}_\theta^*$ satisfies $\delta^\phi(\Gamma)<+\infty$, then $P_\Gamma^\phi(\delta^\phi(\Gamma))=+\infty$.
\end{theorem}

Sambarino \cite[Cor. 5.7.2]{sambarino-dichotomy} proved Theorem \ref{thm: rel Anosov divergent} for Anosov groups (certain special cases were previously established by Lee--Oh \cite[Lem. 7.11]{LeeOh} and by Dey--Kapovich \cite[Thm. A]{DK1}).

Glorieux--Monclair--Tholozan \cite[Cor. 1.2]{GMT} (see also Pozzetti--Sambarino--Weinhard \cite[Prop. 4.1]{PSW}) proved that  if $\Gamma$ is also $P_{1,d-1}$-Anosov, 
then $\delta^{\alpha_1}(\Gamma)$ provides an upper bound for the Hausdorff dimension of $\Lambda_{1,d-1}(\Gamma)$. As a consequence of our shadow lemma, we extend this result to the setting of $P_{k,d-k}$-transverse groups. 

\begin{proposition}[{\cite[Cor 5.2]{CZZ2}}] If $\Gamma\subset\mathsf{SL}(d,\mathbb R)$ is $P_{k,d-k}$-transverse, then 
the Hausdorff dimension of  the conical limit set $\Lambda_{k,d-k}^{con}(\Gamma)$ is at most $\delta^{\alpha_k}(\Gamma)$.
\end{proposition}

It is not true in general, even for $\Psf_{1,d-1}$-Anosov groups, that the Hausdorff dimension of the limit set is given by $\delta^{\alpha_1}$. (Glorieux--Monclair--Tholozan \cite{GMT} show that $\delta^{\omega_H}(\Gamma)$ is a sharp lower bound, where $\omega_H(\vec a)=a_1-a_d)$.
However,
Pozzetti--Sambarino--Wienhard \cite{PSW} introduced a condition guaranteeing a form of equality in the Anosov case.
We say that $\Gamma\subset\mathsf{SL}(d,\mathbb R)$ is \emph{$(1,1,2)$-hypertransverse} if $\Gamma$ is $P_\theta$-transverse, $1,2\in\theta$ and whenever
$F,G,H\in\Lambda_\theta(\Gamma)$ are distinct partial flags, then
$$F^1\oplus G^1\oplus H^{d-2}=\mathbb R^d.$$
Pozzetti--Sambarino--Wienhard \cite[Cor 6.9]{PSW}  showed that if $\Gamma$ is $P_\theta$-Anosov with $1,2\in\theta$ and $(1,1,2)$-hypertransverse,
then  $\delta^{\alpha_1}(\Gamma)$ is the Hausdorff dimension of $\Lambda_{1,d-1}(\Gamma)$. A key step in the proof is an estimate of the inradius (with respect to some Riemannian metric on $\Pb(\Rb^d)$) of the size of shadows in $\Lambda_\theta(\Gamma)$. Combining a generalization of this estimate to the setting of transverse groups with the construction of Bishop--Jones \cite{bishop-jones} we obtained a generalization to
the transverse setting.

\begin{theorem}[{\cite[Thm 8.1]{CZZ2}}]\label{112}
 If $\Gamma\subset\mathsf{SL}(d,\mathbb R)$ is $(1,1,2)$-hypertransverse, 
then $\delta^{\alpha_1}(\Gamma)$ is the Hausdorff dimension of $\Lambda_{1,d-1}^{con}(\Gamma)$.
\end{theorem}

An important class of examples of $(1,1,2)$-hypertransverse subgroups are the images of \emph{cusped Hitchin representations}. These are the representations from a geometrically finite, infinite, discrete subgroup $\Gamma_0\subset\PSL(2,\Rb)$ to $\SL(d,\Rb)$ that admit an equivariant, continuous, positive map from $\Lambda(\Gamma_0)$ to $\Fc_{1,\dots,d-1}$. This notion of a positive map was introduced by Fock--Goncharov~\cite{FG}, and can be thought of as a ``strong monotonicity" condition on the map, see \cite[Section 2.3]{CZZ2} for the precise definitions. In the case when $d=2$, these cusped Hitchin representations are the holonomies of geometrically finite hyperbolic surfaces. One can prove \cite[Theorem 1.4]{CZZ1} that cusped Hitchin representations have finite kernel, and the images of cusped Hitchin representations are $\Psf_{1,\dots,d-1}$-relative Anosov subgroups whose limit set is the image of $\Lambda(\Gamma_0)$ under the positive map. 
Using Theorems \ref{thm: entropy drop} and \ref{112}, we prove the following generalization of a result of Potrie--Sambarino \cite{potrie-sambarino}.

\begin{theorem}[{\cite{CZZ2}}]\label{entropy rigidity}
Let $\rho:\Gamma_0\to\SL(d,\Rb)$ be a cusped Hitchin representation, and let $\Gamma:=\rho(\Gamma_0)$. Then for all $k\in\{1,\dots,d-1\}$, $\delta^{\alpha_k}(\Gamma)\le 1$ and equality holds if and only if $\Gamma_0$ is a lattice in $\PSL(2,\Rb)$.
\end{theorem}

\begin{proof}[Sketch of proof:]
First, one verifies that for all $k\in\{1,\dots,d-1\}$, the image of the induced representation 
\[\bigwedge^k\rho:\Gamma_0\to\SL\left(\bigwedge^k\Rb^d\right)\]
is a $(1,1,2)$-hypertransverse subgroup. Furthermore one can compute that
\[\delta^{\alpha_k}(\Gamma)=\delta^{\alpha_1}\left(\bigwedge^k\rho(\Gamma_0)\right).\]
At the same time, the monotonicity properties of positive maps imply that the image of any positive map has Hausdorff dimension at most $1$. It then follows from Theorem~\ref{112} that $\delta^{\alpha_k}(\Gamma)\le 1$ for all $k\in\{1,\dots,d-1\}$. 

If $\Gamma_0$ is a lattice, then $\Lambda_{1,\dots,d-1}(\Gamma)$ is a continuous curve, and hence $\delta^{\alpha_k}(\Gamma)=1$ for all $k\in\{1,\dots,d-1\}$. Conversely, if $\Gamma_0$ is not a lattice, then by a well-known doubling construction, there is a lattice $\Gamma_0^D\subset\PSL_2(\Rb)$ that contains $\Gamma_0$, and $\Lambda(\Gamma_0)\subset\Lambda(\Gamma_0^D)$ is a proper subset. Furthermore, using properties of positive maps, one can prove that if $\rho:\Gamma_0\to\SL(d,\Rb)$ is a cusped Hitchin representation, then there is a cusped Hitchin representation $\rho^D:\Gamma_0^D\to\SL(d,\Rb)$ such that $\rho^D|_{\Gamma_0}=\rho$. We may then apply Theorem \ref{thm: entropy drop} to deduce that $\delta^{\alpha_k}(\Gamma)<\delta^{\alpha_k}(\Gamma^D)\le 1$, where $\Gamma^D:=\rho^D(\Gamma_0^D)$.
\end{proof}

\medskip\noindent{\bf Remark:} Dey--Kapovich \cite{DK1} showed that if $\Gamma$ is $P_\theta$-Anosov and $\phi\in \mathfrak{a}_\theta^*$ is symmetric, i.e. 
$\phi=i\circ \phi$ and has gradient in the interior of $\mathfrak{a}_\theta^+:=\mathfrak a^+\cap\mathfrak a_\theta$, then $\delta^\phi(\Gamma)$ is the Hausdorff dimension of $\Lambda_\theta(\Gamma)$
with respect to a Gromov pre-metric associated to $\phi$. Dey--Kim--Oh \cite{DKO} showed that one can replace the assumption that
$\phi$ is strictly positive on $\mathfrak{a}_\theta^*-\{\vec 0\}$ with the weaker assumption that the critical exponent is finite. Moreover, they showed that the
assumption that $\phi$ is symmetric is essential.

\section{Flow spaces for transverse groups}

In this section, we will explain how to use our projective geometry tool, Theorem \ref{projvisible}, to construct for any transverse subgroup of $\SL(d,\Rb)$ a flow space that admits a Hopf decomposition. This is used to prove the Hopf--Tsuji--Sullivan dichotomy, which is a theorem that describes the ergodicity properties of Patterson--Sullivan measures for transverse groups.

As before, let $\Gamma\subset\mathsf{SL}(d,\mathbb R)$ be a $P_\theta$-transverse subgroup, and let $\Gamma_0$, $\Omega$, $\rho$ and $\xi$ be as given in Theorem \ref{projvisible}. Since the Hilbert metric on $\Omega$ is Finsler, we may define its unit tangent bundle $T^1\Omega$, which admits a natural geodesic flow $\varphi_s$ and an ${\rm Aut}(\Omega)$-action that commutes with $\varphi_s$. 
The flow $\varphi_s$ can be described as followed: For any $\vec v\in T^1\Omega$, let $g_{\vec v}:\mathbb R\to \Omega$ be the unit speed parameterization of a projective line, so that $g_{\vec v}'(0)=\vec v$. Then for any $s\in\Rb$, $\varphi_s(\vec v)=g_{\vec v}'(s)$.

Let $\tilde U(\Gamma_0)\subset T^1\Omega$ be the set of tangent vectors to lines ending at points in $\Lambda_\Omega(\Gamma_0)$. After choosing a basepoint $b_0\in\Omega$, we  can again define a Hopf parametrization, using horospheres for the Hilbert metric, so that
$$\tilde U(\Gamma_0)=\Lambda_\Omega(\Gamma_0)^{(2)}\times \mathbb R,$$
where $\Lambda_\Omega(\Gamma_0)^{(2)}$ denotes the set of distinct pairs of points in $\Lambda_\Omega(\Gamma_0)$.
In this parametrization
$$\varphi_s(w,z,t)=(w,z,s+t)\quad\text{and}\quad \gamma(w,z,t)=\left(\gamma(w),\gamma(z),t+h_z(\gamma^{-1}(b_0),b_0)\right)$$
for all $\gamma\in\Gamma_0$, where
$$h_z(\gamma^{-1}(b_0),b_0)=\lim_{x_n\to z} d_\Omega(x_n,\gamma^{-1}(b_0))-d_\Omega(x_n,b_0).$$
Notice that $\tilde U(\Gamma_0)\subset T^1\Omega$ is a closed, $\Gamma_0$-invariant, $\varphi_s$-invariant subset, so $\varphi_s$ descends to a geodesic flow on $U(\Gamma_0):=\tilde U(\Gamma_0)/\Gamma_0$.

There is a natural involution $\iota:\mathfrak{a}_\theta^*\to\mathfrak{a}_\theta^*$ so that $\phi(\kappa(A))=\iota(\phi)(\kappa(A^{-1}))$.
In particular, $\iota(\omega_k)=\omega_{d-k}$. We often denote $\iota(\phi)$ by $\bar\phi$.
Since $P_\Gamma^\phi(s)=P_\Gamma^{\bar\phi}(s)$, we see that $\delta^\phi(\Gamma)=\delta^{\bar\phi}(\Gamma)$. Let $\phi\in\mathfrak{a_\theta^*}$ such that $\delta^\phi(\Gamma)<+\infty$, let $\mu$ be a $\phi$-Patterson-Sullivan measure for $\Gamma$ and
let $\bar\mu$ be a $\bar\phi$-Patterson-Sullivan measure for $\Gamma$. We will now describe how one can use $\mu$ and $\bar\mu$ to construct a $\varphi_s$-invariant measure on $U(\Gamma_0)$. 

To do so, we need an analog of the Gromov product. Let $\Fc_\theta^{(2)}$ denote the set of transverse pairs of flags in $\Fc_\theta$. Define the \emph{Gromov product} 
$$G_\theta:\Fc_\theta^{(2)}\to \mathfrak{a}_\theta$$
to be the map such that for all $k\in\theta$,
\[\omega_k(F_1,F_2)=\log\frac{\abs{\det(f_i(v_j))_{i,j}}}{\norm{f_1\wedge\dots\wedge f_k}_k^*\norm{v_1\wedge\dots\wedge v_k}_k}\]
where $v_1,\dots,v_k$ is some (any) basis of $F_1^k$; $f_1,\dots,f_k$ is some (any) basis for 
\[\{f\in(\Rb^d)^*:f(F_2^{d-k})=0\};\] 
$(f_i(v_j))_{i,j}$ denotes the $k\times k$-matrix whose $(i,j)$-th entry is $f_i(v_j)$;
and $\norm{\cdot}_k^*$ and $\norm{\cdot}_k$ are respectively the norms on $\bigwedge^k(\Rb^d)^*$ and $\bigwedge^k\Rb^d$ induced by the standard inner product on $\Rb^d$. One can verify that for all $A\in\SL(d,\Rb)$ and $(F_1,F_2)\in\Fc_\theta^{(2)}$
\[G(A(F_1),A(F_2))-G(F_1,F_2)=-\hat\iota\circ B_\theta(A,F_1)-B_\theta(A,F_2),\]
where  $\hat\iota:\mathfrak a_\theta\to\mathfrak a_\theta$  is the involution  $\hat\iota(a_1,a_2,\ldots,a_d)=(a_d,a_{d-1},\ldots, a_1)$.
 It follows that
$$d\tilde m(w,z,t):=e^{-\delta^\phi(\Gamma)\phi(G(\xi(w),\xi(z))}d\bar\mu(w) d\mu(z) ds(t)$$
is a $\Gamma_0$-invariant measure on $\tilde U(\Gamma_0)$. (Here $ds$ is Lebesgue measure on $\mathbb R$, so $d\tilde m$ is also $\varphi_s$-invariant.) This descends to a $\varphi_s$-invariant measure on $U(\Gamma_0)$, known as the \emph{BMS-measure}.

If $P_\Gamma^\phi(\delta^\phi(\Gamma))=+\infty$, then we can again adapt Hopf's argument
to show that the geodesic flow on $U(\Gamma_0)$ is ergodic with respect to the quotient $dm$ of $d\tilde m$. We then obtain the following analog of the Hopf--Tsuji--Sullivan dichotomy in our setting.

\begin{theorem}[{\cite[Thm. 1.4]{CZZ3}}] 
\label{HST:transverse}
 Suppose $\Gamma\subset\mathsf{SL}(d,\mathbb R)$ is a non-elementary $P_\theta$-transverse subgroup, 
$\phi\in \mathfrak{a}^*_\theta$ and $\delta:=\delta^\phi(\Gamma) < +\infty$. 
\begin{itemize}
\item If 
$Q_\Gamma^\phi(\delta)=+\infty$, 
then there exists a unique $\phi$-Patterson--Sullivan measure $\mu$ for $\Gamma$ and a unique $\bar{\phi}$-Patterson--Sullivan measure $\bar{\mu}$  for $\Gamma$. Moreover, $\mu( \Lambda_\theta^{\rm con}(\Gamma)) = 1=\bar\mu( \Lambda_\theta^{\rm con}(\Gamma))$, the action of $\Gamma$ on $(\Lambda_\theta(\Gamma)^2, \bar\mu \otimes \mu)$ is conservative and ergodic, and the actions of $\Gamma$ on $(\Lambda_\theta(\Gamma), \mu)$ and $(\Lambda_\theta(\Gamma), \bar\mu)$ are both ergodic. 
\item If $Q_\Gamma^\phi(\delta) < +\infty$, then for any $\phi$-Patterson--Sullivan measure $\mu$ for $\Gamma$ and any $\bar{\phi}$-Patterson--Sullivan measure $\bar{\mu}$ for $\Gamma$, $\mu( \Lambda_\theta^{\rm con}(\Gamma)) = 0 = \bar{\mu}( \Lambda_\theta^{\rm con}(\Gamma)) $, and the action of $\Gamma$ on $(\Lambda_\theta(\Gamma)^2, \bar\mu \otimes \mu)$ is dissipative and not ergodic. 
\end{itemize}
\end{theorem}

Dey--Kapovich \cite{DK1} proved that when $\Gamma$ is $\Psf_\theta$-Anosov and $\phi\in\mathfrak a_\theta^*$ is positive on $\mathfrak a^+\cap\mathfrak a_\theta$, then there is a unique $\phi$-Patterson--Sullivan measure $\mu$, $\mu(\Lambda_\theta(\Gamma))=1$, and the action of $\Gamma$ on $(\Lambda_\theta(\Gamma),\mu)$ is ergodic. Sambarino \cite{sambarino-dichotomy} later used thermodynamical formalism to strengthen Dey--Kapovich's results to allow for any $\phi\in\mathfrak a_\theta^*$ such that $\delta^\phi(\Gamma)<+\infty$, and he also established the ergodicity of the action of $\Gamma$ on $(\Lambda_\theta(\Gamma)^2, \bar\mu \otimes \mu)$ is ergodic. In the case when $\Gamma$ is Zariski dense and $\Psf_{1,\dots,d-1}$-Anosov, these results were also proven by Lee--Oh \cite{LeeOh}.

Combining Theorem \ref{HST:transverse} with Theorem \ref{thm: rel Anosov divergent}, we see that for every $\Psf_\theta$-relatively Anosov subgroup $\Gamma\subset\SL(d,\Rb)$ and every $\phi\in\mathfrak a_\theta^*$ such that $\delta^\phi(\Gamma)<+\infty$, there is a unique $\phi$-Patterson--Sullivan measure for $\Gamma$, which is ergodic with respect to the $\Gamma$-action on $\Lambda_\theta(\Gamma)$. 

We may also apply Theorem \ref{HST:transverse} to establish a concavity result for the critical exponent regarded as a function on
$\mathfrak{a}_\theta^*$.

\begin{theorem}[{\cite[Thm. 13.1]{CZZ3}}]
\label{manhattan}
Suppose $\Gamma\subset\mathsf{SL}(d,\mathbb R)$ is a non-elementary $P_\theta$-transverse subgroup, $\phi_1, \phi_2 \in \mathfrak{a}^*_\theta$ and
$\delta^{\phi_1}(\Gamma)=\delta^{\phi_2}(\Gamma) =1$. 
If $\phi = \lambda \phi_1+(1-\lambda)\phi_2$ where $\lambda \in (0,1)$, then 
$$
\delta^{\phi}(\Gamma) \le 1. 
$$
Moreover, if  $\delta^{\phi}(\Gamma) = 1$ and $Q_\Gamma^\phi$ diverges at its critical exponent, then $\ell^{\phi_1}(\gamma) = \ell^{\phi_2}(\gamma)$ 
for all $\gamma \in \Gamma$.
\end{theorem} 

One may then apply a result of Benoist \cite{benoist-limit-cone} to conclude that strict concavity holds whenever $\Gamma$ is Zariski dense.

\begin{corollary}[{\cite[Cor. 13.2]{CZZ3}}]
Suppose $\Gamma\subset\mathsf{SL}(d,\mathbb R)$ is Zariski dense and  $P_\theta$-transverse, $\phi_1, \phi_2 \in \mathfrak{a}^*_\theta$, $\phi_1\ne\phi_2$
and $\delta^{\phi_1}(\Gamma)=\delta^{\phi_2}(\Gamma) =1$.
If $\phi = \lambda \phi_1+(1-\lambda)\phi_2$ where $\lambda \in (0,1)$ and $Q_\Gamma^\phi$ diverges at its critical exponent, then $\delta^{\phi}(\Gamma) <1$.
\end{corollary}

\medskip\noindent
{\bf Remark:} Potrie--Sambarino \cite{potrie-sambarino} showed that if $\phi\in\mathfrak{a}^*$ is strictly positive on $\mathfrak{a}^+-\{0\}$, then
then the crticial exponent $\delta^\phi$, regarded as  a function on a Hichin component achieves its maximum precisely on the Fuchsian locus.
One can use Theorerms \ref{entropy rigidity} and \ref{manhattan}  to establish the same result for cusped Hitchin representations of Fuchsian
lattices (see \cite{CZZ2} for a more detailed statement).

\section{Recent work}

In this section, we survey some of the work which has been announced more recently which generalizes and expands on the work 
described in the previous sections.

\subsection{Kim, Oh and Wang's approach to dynamics of transverse group}
\label{KOWsec}
Kim, Oh and Wang \cite{KOW} developed an alternate approach to studying Patterson--Sullivan measures associated to transverse groups. 
They first define an abstract Weyl chamber flow
$$\tilde W(\Gamma)=\Lambda_\theta(\Gamma)^{(2)}\times \mathfrak{a}_\theta$$
where an element of $\mathfrak a_\theta$ acts trivially on the first two factors and by addition on the final factor.
The group $\Gamma$ acts on $\tilde W(\Gamma)$ by
$$\gamma(F,G,\vec a)=\left(\gamma(F),\gamma(G), \vec a+B_\theta(\gamma, G)\right).$$
They prove that $\Gamma$ acts properly discontinuously and the $\mathfrak a_\theta$ flow on $\tilde W(\Gamma)$ descends
to a $\mathfrak a_\theta$ flow on $W(\Gamma)=\tilde W(\Gamma)/\Gamma.$ 

If $\phi\in\mathfrak{a}_\theta^*$ and $\delta^\phi(\Gamma)<+\infty$, then there is a projection
$$p^\phi:\tilde W(\Gamma)\to \tilde W^\phi(\Gamma):=\Lambda_\theta(\Gamma)^{(2)}\times \mathbb R$$
given by
$$p^\phi(F,G,\vec a)=(F,G,\phi(\vec a)).$$
The $\Gamma$ action of $\tilde W(\Gamma)$ descends to an action on $\tilde W^\phi(\Gamma)$ which satisfies 
$$\gamma(F,G, t)=\Big(\gamma(F),\gamma(G), t+\phi(B_\theta(\gamma, G))\Big)$$
and commutes with the flow $\varphi_s$ on $\tilde W^\phi(\Gamma)$ given by $\varphi_s(F,G,t)=(F,G,s+t)$. They show that $\Gamma$ is $P_\theta$-Anosov if and only if 
\[W^\phi(\Gamma):=\tilde W^\phi(\Gamma)/\Gamma\] 
is compact. Furthermore, if $\delta^\phi(\Gamma)<+\infty$, they construct a BMS measure on $W^\phi(\Gamma)$ and show that if
$P^\phi_\Gamma(\delta^\phi(\Gamma))=+\infty$, then the flow $\varphi_s$ on $W^\phi(\Gamma)$ is ergodic with
respect to this BMS measure (which is unique in this case). 
Their flow space $W^\phi(\Gamma)$ has an important advantage over the flow space $U(\Gamma_0)$: The periods of the flow on $W^\phi(\Gamma)$ are exactly given by 
$$\left\{\ell^\phi(\gamma): \gamma\in [\Gamma]_{hyp}\right\}.$$

They use their flow space  to recover the main results in~\cite{CZZ3} for Zariski dense transverse groups. They also obtain new applications to the study
of the $\theta$-growth indicator function for $\theta$-transverse groups. Moreover, they establish an analogue of Ahlfors' Measure conjecture for $\Psf_\theta$-Anosov
groups, by showing that the $\theta$-limit set either has Lebesgue measure zero or is all of $\mathcal F_\theta$, in which case the group is a uniform lattice
in a rank one Lie group (for Zariski dense, Borel Anosov subgroups, this was proven by Lee and Oh \cite{LeeOh1}).

\subsection{A general theory of Patterson-Sullivan measures for convergence group actions}
Blayac--Canary--Zhu--Zimmer \cite{BCZZ1,BCZZ2} develop a theory of Patterson--Sullivan measures for general cocycles associated to convergence group 
actions. In~\cite{BCZZ1}, the authors developed a Patterson--Sullivan theory for coarse-cocycles with a certain uniform expansion property 
(called \emph{expanding coarse-cocyles}) and in~\cite{BCZZ2} the authors consider cocycles that are a part of a continuous GPS system 
(which consists of a pair of expanding coarse cycles linked by a coarse Gromov product). 
We will describe some of their results for the less technical continuous GPS system setting.

Suppose $\Gamma$ acts as a convergence group on a compact perfect metric space $X$. A continuous map $\sigma:\Gamma\times X\to \mathbb R$ is a \emph{cocycle} when
$$\sigma(\gamma_1\gamma_2,x)=\sigma(\gamma_1,\gamma_2x)+\sigma(\gamma_2,x)$$
for all $\gamma_1,\gamma_2\in\Gamma$ and $x\in X$. The \emph{$\sigma$-period} of an element $\gamma \in \Gamma$ is 
$$
\ell_\sigma(\gamma) := \begin{cases} 
\sigma(\gamma, \gamma^+) & \text{if } \gamma \text{ is loxodromic and $\gamma^+$ is the attracting fixed point of $\gamma$,} \\
0 & \text{otherwise.}
\end{cases}
$$ 
It is natural to assume that $\sigma$ is {\em proper}, that is 
$$
\lim_{n \to \infty} \ell_\sigma(\gamma_n)=+\infty
$$
whenever $\{\gamma_n\} \subset \Gamma$ is a sequence of distinct loxodromic elements whose repelling/attracting fixed points $\{ (\gamma_n^-, \gamma_n^+)\}$ are relatively compact in $X^{(2)} \subset X \times X$, the space of distinct pairs.  

Given a proper cocycle $\sigma : \Gamma \times X \rightarrow \Rb$, a \emph{$\sigma$-Patterson--Sullivan measure of dimension $\delta$} is a probability measure 
$\mu$ on $X$ where for all $\gamma\in\Gamma$, the measures $\mu$ and $\gamma_*(\mu)$ are absolutely continuous, and 
$$
\frac{d\gamma_*(\mu)}{d\mu}=e^{-\delta\sigma(\gamma^{-1},\cdot)}
$$
almost everywhere. 

A cocycle $\sigma$ is {\em expanding} if it is proper and we may define a magnitude $||\gamma||_\sigma$ for all $\gamma\in\Gamma$ so
that for all $\epsilon>0$ there exists $C>0$ so that if $d(x,\gamma^{-1})>\epsilon$ (for an appropriate metric on $\Gamma\cup X$), then
$$||\gamma||_\sigma -C\le \sigma(\gamma,x)\le ||\gamma||_\sigma+C.$$
We can then define a Poincar\'e series
$$P_\sigma(s)=\sum_{\gamma\in\Gamma}e^{-s||\gamma||_\sigma}$$
with an associated critical exponent $\delta_\sigma(\Gamma)$. If the critical exponent is finite,  then there exists a Patterson-Sullivan measure on $X$ of
dimension $\delta_\sigma(\Gamma)$ (see \cite[Prop 4.1]{BCZZ1}). One can establish an analogue of Sullivan's Shadow Lemma in this setting
and derive many of its standard consequences (see \cite{BCZZ1}).

A continuous {\em GPS}-system for $\Gamma$ is a pair of proper  continuous cocyles $\sigma$ and $\bar\sigma$ related by a {\em Gromov product} which is
a continuous function $G:X^{(2)}\to \mathbb R$ so that
$$\bar\sigma(\gamma,x)+\sigma(\gamma,y)=G(\gamma(x),\gamma(y))-G(x,y).$$
The cocycles in a GPS system are always expanding and $\delta_\sigma(\Gamma)=\delta_{\bar\sigma}(\Gamma)$ (see \cite[Prop 3.4]{BCZZ1}).
For $\Hb^d$, the Busemann cocycle $B$ and Gromov product $G$ form a GPS system $(B,B,G)$ for any discrete subgroup in $\mathsf{Isom}(\Hb^d)$. 
More generally, for transverse subgroups, the cocycles and Gromov product introduced earlier in this paper form a GPS system.

When $(\sigma, \bar \sigma, G)$ is a GPS system for $\Gamma$ acting on $X$, a properly discontinuous action of $\Gamma$ on $X^{(2)} \times \Rb$ can be defined by 
$$
\gamma \cdot (x,y,t) = (\gamma x, \gamma y, t + \sigma(\gamma, y))
$$
(see \cite[Prop.\,10.2]{BCZZ1}). Then if $\mu$ is a $\sigma$-Patterson--Sullivan measure of dimension $\delta=\delta_\sigma(\Gamma)$ and 
$\bar\mu$ is a $\bar\sigma$-Patterson--Sullivan measure of dimension $\delta$, the measure 
$$
d\tilde m(w,z,t): = e^{\delta G(x,y)} d\bar\mu(x)d\mu(y)ds(t)
$$
on $X^{(2)} \times \Rb$ is $\Gamma$-invariant (this follows immediately from the defining relationship between $\bar\sigma$, $\sigma$, and $G$). As usual we can define a flow $\varphi_s$ on $X^{(2)} \times \Rb$ by translating on the $\Rb$ factor. This flow descends to a flow on the quotient $(X^{(2)} \times \Rb) / \Gamma$ and the measure $\tilde m$ descends to a flow invariant measure, which we call a \emph{Bowen--Margulis--Sullivan (BMS) measure}. 
In the special case of GPS systems associated to transverse groups this construction coincides with the one in~\cite{KOW}, see subsection~\ref{KOWsec} above. 

The following analogue of the Hopf--Tsuji--Sullivan dichotomy was established in \cite{BCZZ1}.

\begin{theorem}[{\cite[Thm. 1.8 and Section 10.4]{BCZZ1}}]
Suppose $(\sigma, \bar\sigma, G)$ is a continuous GPS system and $\delta_\sigma(\Gamma) < +\infty$. 
Let $\mu$ and $\bar \mu$ be Patterson--Sullivan measures of dimension $\delta$ for $\sigma$ and  $\bar{\sigma}$.
\begin{enumerate} 
\item If $\sum_{\gamma \in \Gamma} e^{-\delta \norm{\gamma}_\sigma} = +\infty$, then:
\begin{enumerate}
\item $\delta = \delta_\sigma(\Gamma)$.
\item $\mu( \Lambda^{\rm con}(\Gamma)) = 1 = \bar\mu(\Lambda^{\rm con}(\Gamma))$. 
\item The natural flow on $X^{(2)} \times \Rb / \Gamma$ is ergodic with respect to the Bowen--Margulis--Sullivan measure constructed from $\bar\mu$ and $\mu$. 
\item The $\Gamma$ action on $(X^{(2)},\bar\mu \otimes \mu)$ is ergodic and conservative. 
\end{enumerate} 
\item If $\sum_{\gamma \in \Gamma} e^{-\delta \norm{\gamma}_\sigma} < +\infty$, then:
\begin{enumerate}
\item $\delta \ge \delta_\sigma(\Gamma)$.
\item $\mu( \Lambda^{\rm con}(\Gamma)) = 0 = \bar\mu(\Lambda^{\rm con}(\Gamma))$. 
\item The natural flow on $X^{(2)} \times \Rb / \Gamma$ is non-ergodic with respect to the Bowen--Margulis--Sullivan measure constructed from $\bar\mu$ and $\mu$. 
\item The $\Gamma$ action on $(X^{(2)}, \bar\mu \otimes \mu)$ is non-ergodic and dissipative. 
\end{enumerate} 
\end{enumerate} 

\end{theorem} 

In \cite{BCZZ2}, we establish mixing and equidistribution results for continuous GPS systems  with finite BMS measure. Describing these
results would take us to far afield. However, we will state a counting result for continuous GPS systems for geometrically finite convergence
group actions which we derive as a consequence of these results.

The {\em length spectrum} $\mathcal L(\sigma,\bar\sigma,G)$ of
a GPS system can be defined to be
$$
\mathcal L(\sigma,\bar\sigma,G):=\left\{\ell_\sigma(\gamma)+\ell_{\bar\sigma}(\gamma): \gamma\in\Gamma\text{ is loxodromic}\right\}.
$$ 
We say that $\mathcal L(\sigma,\bar\sigma,G)$ is {\em non-arithmetic} if it generates a dense subgroup of $\mathbb R$.
Let $[\Gamma_{\rm lox}]$ denote the set of conjugacy classes of loxodromic elements.

\begin{theorem}[{\cite[Thm. 1.1]{BCZZ2}}]\label{thm:GPS counting}
Suppose $(\sigma, \bar{\sigma}, G)$ is a continuous GPS system for a torsion-free geometrically finite convergence group $\Gamma\subset\mathsf{Homeo}(X)$ where $\delta:=\delta_\sigma(\Gamma)<+\infty$. If 
\begin{enumerate}
\item $\mathcal L(\sigma,\bar\sigma,G)$ is non-arithmetic and 
\item  $\delta_\sigma(P)<\delta$ for all maximal parabolic subgroups $P \subset \Gamma$, 
\end{enumerate}
then
$$
\#\{[\gamma]\in [\Gamma_{\rm lox}] : 0<\ell_\sigma(\gamma) \le R\} \sim \frac{e^{\delta R}}{\delta R},
$$
i.e.\ the ratio of the two sides goes to 1 as $R \to +\infty$.
\end{theorem}

\subsection{Applications to transverse and relatively Anosov groups}
We first observe that we may place transverse groups into the context of the previous subsection.

\begin{proposition}[{\cite[Prop 10.3]{BCZZ2}}]
Suppose $\Gamma$ is non-elementary $\mathsf{P}_\theta$-transverse, $\phi\in \mathfrak{a}_\theta^*$  and 
$$
\lim_{n \rightarrow \infty} \phi(\kappa(\gamma_n)) = +\infty
$$
whenever $\{\gamma_n\}$ is a sequence of distinct elements of $\Gamma$. Define cocycles $\sigma_\phi, \bar{\sigma}_\phi \colon \Gamma \times \Lambda_\theta(\Gamma) \to \Rb$ by 
$$
\sigma_\phi(\gamma, F) = \phi(B_\theta(\gamma,F)) \quad \text{and} \quad \bar\sigma_\phi(\gamma, F) = \iota^*(\phi)(B_\theta(\gamma,F)).
$$
Then $(\sigma_\phi,\bar\sigma_{\phi}, \phi\circ G)$ is a continuous  GPS system for the action of $\Gamma$ on $\Lambda_\theta(\Gamma)$. Moreover:
\begin{enumerate}
\item  One can choose magnitude functions 
$$
\norm{\gamma}_{\sigma_\phi} = \phi(\kappa(\gamma)) \quad \text{and} \quad \norm{\gamma}_{\bar\sigma_\phi} =  \iota^*(\phi)(\kappa(\gamma)).
$$
 In particular, $\delta^\phi(\Gamma) = \delta_{\sigma_\phi}(\Gamma)$. 
\item  If $\lambda \colon \mathsf G \to \mathfrak a^+$ is the Jordan projection, then 
$$
\phi(\lambda(\gamma)) = \ell_{\sigma_\phi}(\gamma)
$$
for all $\gamma \in \Gamma$. 
\end{enumerate}
\end{proposition} 

One can use standard arguments to show that, in the context of transverse groups, the length spectrum is always non-arithmetic~\cite[Prop. 10.4]{BCZZ2}.
Hence, combining Theorem~\ref{thm:GPS counting} with the results in~\cite{CZZ4} (see the discussion before Theorem~\ref{thm: rel Anosov divergent}) one 
obtains the following counting result for relatively Anosov groups.  One also obtains mixing and equidistribution results in this setting.

\begin{theorem}[{\cite[Cor. 10.7]{BCZZ2}}]
\label{counting}
Suppose $\Gamma$ is torsion-free and $\mathsf{P}_\theta$-relatively Anosov. If 
$\phi\in\mathfrak{a}_\theta^*$ and $\delta:=\delta^\phi(\Gamma)<+\infty$, then
$$
\#\{[\gamma]\in [\Gamma_{\rm lox}] :0<\phi(\lambda(\gamma))  \le R\} \sim \frac{e^{\delta R}}{\delta R}.
$$
\end{theorem}

Sambarino \cite{sambarino-quantitative} (see also \cite{BCLS}) established Theorem \ref{counting} for Anosov groups.

\medskip

Kim--Oh~\cite{KO}  used the approach  of Kim--Oh--Wang to show that if $\delta^\phi(\Gamma)<+\infty$ and $\Gamma$ is $\Psf_\theta$-relatively Anosov, 
then the BMS measure on $W^\phi(\Gamma)$ is finite and the flow is mixing. They also show that the BMS measure is the unique measure
of maximal entropy in this settting.

\subsection{Other work}
Delarue--Monclair--Sanders \cite{DMS} recently showed that flows associated to $P_{1,d-1}$-Anosov subgroups of $\mathsf{SL}(d,\mathbb R)$ are
exponentially mixing. They obtain applications to the associated zeta functions and obtain stronger asymptotic counting results (when the group
is irreducible).  Chow--Sarkar \cite{CS} generalized their work to the setting of Zariski dense $\Psf_\theta$-Anosov groups.

Kim--Oh--Wang \cite{KOW23} proved an analogue of the Hopf-Tsuji-Sullivan dichotomy for the  action of subspaces of $\mathfrak a_\theta$ on $W(\Gamma)$. 

Li--Pan--Xu \cite{LPX} showed that the Hausdorff dimension of the limit set of  an Anosov subgroup of $\mathsf{SL}(3,\mathbb R)$ is given by its affinity exponent.
One surprising consequence is that the Hausdorff dimension varies discontinuously over the space of Anosov representations of a surface group into
$\mathsf{SL}(3,\mathbb R)$.

\end{document}